\documentclass[12pt,reqno]{amsart}

\usepackage{float}
\usepackage[foot]{amsaddr}

\usepackage[nopar]{lipsum}
\usepackage{stmaryrd}
\usepackage{amsmath, amscd,amsxtra,amssymb,mathrsfs, bbm}

\usepackage[initials, alphabetic, msc-links, nobysame,short-publishers]{amsrefs}

\usepackage[showframe=false,marginparsep=0in,
top=0.9in, 
text={5.85in,9.3in},
left=1.325in, 
bindingoffset=0in,twoside=true]{geometry} %

\usepackage{tikz}
\usepackage{wrapfig}
\usepackage{enumitem}
\usepackage{arydshln}
\usepackage{graphicx}
\usepackage{float}
\usepackage{subcaption}
\usetikzlibrary{arrows,decorations.pathmorphing,backgrounds,positioning,fit,petri,graphs,automata, intersections}

\usepackage{kbordermatrix}
\setlist[itemize]{leftmargin=15pt}
\setlist[enumerate]{leftmargin=*}

\makeatletter
\def\subsubsection{\@startsection{subsubsection}{3}%
  \z@{.5\linespacing\@plus.7\linespacing}{.1\linespacing}%
  {\normalfont\bfseries}}
\makeatother

\makeatletter
\def\subsection{\@startsection{subsection}{2}%
  \z@{.5\linespacing\@plus.7\linespacing}{.1\linespacing}%
  {\normalfont\bfseries}}
\makeatother

\usepackage{parskip}

\usepackage{etoolbox}
\makeatletter
\patchcmd{\chapter}
{\if@openright\cleardoublepage\else\clearpage\fi}{}{}{}
\makeatother
\newenvironment{td}[0]
{\begin{tikzcd}[ampersand replacement=\&, cells={outer sep=2pt, inner sep=2pt},
]
}		
	{\end{tikzcd}}

\newenvironment{gluing-diag}[0]
{\begin{tikzcd}[ampersand replacement=\&, cells={outer sep=1pt, inner sep=0pt}, row sep=15pt, column sep=15pt, font={\small},
labels={rectangle, font={\small}, inner sep=0pt, outer sep=1pt, minimum size=0cm,yshift=1pt},
arrow style=math font,
		arrows={->,semithick, >=stealth'},
baseline=0pt
]}		
	{\end{tikzcd}}

\newenvironment{gluing-diag2}[0]
{\begin{tikzcd}[ampersand replacement=\&, cells={outer sep=1pt, inner sep=0pt}, row sep=5pt, column sep=33pt, font={\small},
labels={rectangle, font={\small}, inner sep=0pt, outer sep=1pt, minimum size=0cm,yshift=1pt},
arrow style=math font,
		arrows={->,semithick, >=stealth'},
baseline=0pt
]}		
	{\end{tikzcd}}

\newenvironment{rg}[0]{
\begin{tikzcd}[ampersand replacement=\&, baseline=0cm, row sep=1cm,
cells={nodes={draw=black!80, fill=white, circle, outer sep=1pt, inner sep=0pt, minimum size=5pt} },
		arrows={-,  draw=black!25, line width=2pt, inner sep=0pt }]
}
	{\end{tikzcd}}

\newenvironment{qd-rel}[0]
{\qdlayout}{\end{tikzcd}}

	{\end{tikzcd}}

	{\end{tikzcd}}

	{\end{tikzcd}}

\newcommand{\qdlayout}[0]
{
	\begin{tikzcd}
		[ampersand replacement=\&, column sep=1.25cm,
		cells={rectangle, inner sep=0pt, outer sep=0pt, minimum width=1.0cm, minimum height=0.625cm},
		labels={rectangle, font=\normalsize, inner sep=0pt, outer sep=2pt, minimum size=0cm},
		arrows={->, semithick, >=stealth'}]
	}

	\newcommand{\Tau}{\mathcal{T}}

	\newcommand{\CM}{ {M}^{\bt}}

	\newcommand{\CP}{ {P}^{\bt}}

	\newcommand{\nuu}{{\varepsilon^\star}}
	\newcommand{\olnuu}{{\ol{\varepsilon}}}

	\newcommand{\kk}{\mathbbm{k}}
	
	\newcommand{\KK}{\mathrm{K}}
	
	\newcommand{\mx}{\mathfrak m}

	\newcommand{\Rx}{\mathrm R}
	
	\newcommand{\Sx}{\mathrm S}

	\newcommand{\A}{\Lambda} %
	\newcommand{\AI}{{\mathrm A}} 
	 
	\newcommand{\BI}{{\mathrm B}}

	\DeclareMathOperator{\D}{D}
	\newcommand{\Db}{\D^{\mathrm{b}}}

	\newcommand{\Dbfd}{\D^{\mathrm{b}}_{\mathrm{fd}}}

	\DeclareMathOperator{\Hom}{Hom}

	\DeclareMathOperator{\id}{id}
	\DeclareMathOperator{\Id}{Id}

	\DeclareMathOperator{\soc}{soc}
	\DeclareMathOperator{\modd}{mod}

	\newcommand{\md}{
\modd}

\usepackage{graphicx}

\makeatletter
\newcommand*\bigcdot{\mathpalette\bigcdot@{.75}}
\newcommand*\bigcdot@[2]{\mathbin{\vcenter{\hbox{\scalebox{#2}{$\m@th#1\bullet$}}}}}
\makeatother

\newcommand*{\setumath}[2]{%
  \csname luatexUmath#1\endcsname\displaystyle=#2\relax
  \csname luatexUmath#1\endcsname\luatexcrampeddisplaystyle=#2\relax
  \csname luatexUmath#1\endcsname\textstyle=#2\relax
  \csname luatexUmath#1\endcsname\luatexcrampedtextstyle=#2\relax
  \csname luatexUmath#1\endcsname\scriptstyle=#2\relax
  \csname luatexUmath#1\endcsname\luatexcrampedscriptstyle=#2\relax
  \csname luatexUmath#1\endcsname\scriptscriptstyle=#2\relax
  \csname luatexUmath#1\endcsname\luatexcrampedscriptscriptstyle=#2\relax
}

\usepackage{mathtools}
\newcommand*{\myov}[1]{\overbracket[1pt][-1.75pt]{#1}}

	\newcommand{\bt}{\bigcdot}
	\newcommand{\wt}{\widetilde}
	
	\newcommand{\ol}{\myov}
	\newcommand{\mt}{\mathtt}

	\newcommand{\UG}{{\mathrm{G}}}

	\newcommand{\omegac}{{\A\!\mathstrut}^{\vee}} %
	\newcommand{\omegad}{{\AI}^{*}} %

	\DeclareMathOperator{\Mat}{Mat}
	\newcommand{\N}{\mathbb{N}}

	\newcommand{\R}{\mathbb{R}}
	\newcommand{\C}{\mathbb{C}}
	\newtheorem{thm}{Theorem}[section]
	
	\newtheorem{cor}[thm]{Corollary}
	\newtheorem{prp}[thm]{Proposition}
	\newtheorem{lem}[thm]{Lemma}

	\newtheorem{dfn}[thm]{Definition}
	\newtheorem{constr}[thm]{Construction}

	\newtheorem{dfn-sp}[thm]{Definition*}
	\newtheorem{ex}[thm]{Example}
	\newtheorem{ex-sp}[thm]{Example*}
	\newtheorem{rmk}[thm]{Remark}
	
	\newtheorem{notation}[thm]{Notation}
	\newtheorem*{sketch-proof}{Sketch of proof}
	\newtheorem*{comment-proof}{Comment on proof}
\newtheorem{intro-thm}[]{Theorem}[]

\newtheorem{manualtheoreminner}{Theorem}
\newenvironment{manualtheorem}[1]{%
  \renewcommand\themanualtheoreminner{#1}%
  \manualtheoreminner
}{\endmanualtheoreminner}

\newtheorem{manualcorollaryinner}{Corollary}
\newenvironment{manualcorollary}[1]{%
  \renewcommand\themanualcorollaryinner{#1}%
  \manualcorollaryinner
}{\endmanualcorollaryinner}

	\DeclareMathOperator{\Perf}{\mathsf{Perf}}
	\DeclareMathOperator{\Perffd}{\mathsf{Perf}_{fd}}
	\DeclareMathOperator{\rk}{rk}
	\DeclareMathOperator{\rad}{\mathsf{rad}}
	
	\DeclareMathOperator{\chr}{\mathsf{char}}

\usepackage{multirow}

\usepackage{enumitem}
\usepackage{longtable}
\usepackage{tikz}
\usepackage{tikz-cd}

\usepackage{color}

\usepackage[utf8]{inputenc}
\usetikzlibrary{calc, positioning, shapes, fit, matrix, decorations}
\usetikzlibrary{decorations.shapes}

\usepackage[all]{xy}
\xyoption{arc}

\usepackage[backgroundcolor=yellow, colorinlistoftodos,disable]{todonotes}

\makeatletter
\providecommand\@dotsep{5}
\def\listtodoname{List of Todos}
\def\listoftodos{\@starttoc{tdo}\listtodoname}
\makeatother

\newcounter{todocounter}

\usepackage{pifont}
\usepackage{ascii}

\newcommand{\blackdiamond}[1][fill=black]{\tikz [x=1ex,y=1ex,line width=.1ex,line join=round, yshift=-0.285ex] \draw  [#1]  (0,.5) -- (.5,1) -- (1,.5) -- (.5,0) -- (0,.5) -- cycle;}%

\address{
Faculty of Mathematics,
Ruhr University Bochum,
Universit\"atsstraße 150,
\mbox{\quad \, 44780 Bochum,
Germany}
}
\email{\href{mailto:Wassilij.Gnedin@ruhr-uni-bochum.de}{Wassilij.Gnedin@ruhr-uni-bochum.de}}
\author{Wassilij Gnedin}
\begin{document}
\title{Calabi-Yau properties of ribbon graph orders}

\begin{abstract}
We 
pursue the order-theoretic approach to ribbon graphs initiated by Kauer and Roggenkamp.
We show that any ribbon graph order is
twisted $1$-Calabi-Yau in general and $1$-Calabi-Yau if the ribbon graph is bipartite. We derive
analogous results for anti-commutative versions of Brauer graph algebras.
\end{abstract}
\maketitle

The Calabi-Yau property of a triangulated category was introduced by Kontsevich.
It has been studied extensively in representation theory, particularly motivated by its connection to cluster theory (see Keller's survey \cite{Keller}).

In many recent works, triangulated Calabi-Yau categories were
associated to triangulated surfaces
(see \cite{Ladkani2017} and references therein).
In homological mirror symmetry, Haiden, Katzarkov and Kontsevich
showed that the partially wrapped Fukaya category of a marked surface with boundary is equivalent to the derived category of a graded finite-dimensional gentle algebra \cite{HKK}.
Vice versa, 
several authors
constructed marked surfaces from finite-dimensional 
as well as infinite-dimensional
gentle algebras (\cites{LP, OPS} respectively \cite{PPP2}).

Earlier, 
 Kauer and Roggenkamp associated to any ribbon graph an \emph{order} which is typically a
non-commutative non-Artinian 
ring 
\cite{Kauer-Roggenkamp}.
Prototypes of their construction
were integral orders of finite groups with cyclic defect.
Since any ribbon graph embeds in a closed, oriented surface,
 such orders have some geometric flavour.

In this paper, we study
 a quiver-theoretic version of the construction by Kauer and Roggenkamp
which we call \emph{ribbon graph orders}.
Like the nodal curve singularity, any ribbon graph order is infinite-dimensional and has infinite global dimension.
By work of Burban and Drozd \cite{Burban-Drozd}, any ribbon graph order is derived-tame.

Similar to the fact that any surface with boundary is some cut of some closed surface, any finite-dimensional
gentle algebra is some quotient of some ribbon graph order.
Presumably, there is a similar connection between
the derived representation theory of a ribbon graph order
and the symplectic geometry of the closed surface of its ribbon graph as in the finite-dimensional setup \cite{LP}*{Remark 3.5.5}.

This is our motivation to study the question whether a ribbon graph order is {symmetric}, or, equivalently, its derived category has a certain $1$-Calabi-Yau property.
An analogous problem was solved by Ladkani
for finite-dimensional Jacobian algebras of triangulated closed surfaces
  \cite{Ladkani2012}.

Since ribbon graph orders are equicharacteristic versions of some group orders, we expected them to be symmetric as well. 
It turns out that this is not true in general.

In Section \ref{sec:basic} of this paper, we recall the notion of a 
\emph{complete
gentle} quiver $(Q,I)$. 
The arrow ideal completion  $\A$ of its path algebra $\kk Q/I$
is 
 a {ribbon graph order}.
We recall that
the derived category $\D^-(\md \A)$ admits a \emph{relative Serre functor} $\mathbb{S}_\A$,
which 
is the  composition
of the derived Nakayama functor $\nu_\A$ and the shift $[1]$.

In Section \ref{sec:canonical} we prove 
the main result of this paper which is an explicit description of the  functor $\nu_\A$.
It can be summarized as follows:

\begin{manualtheorem}{A}[Corollary \ref{cor:A}]\label{thm:1}
The functor $\nu_\A$ is induced by some involution $\nuu$ of the order $\A$.
So
there are isomorphisms  of functors  $\mathbb{S}_\A \cong (\mathstrut_{\text{\textemdash}})_{\scriptsize \nuu} \circ [1]$ and $\mathbb{S}_\A^2 \cong [2]$.
\end{manualtheorem}
In this sense, the order $\A$ is \emph{twisted $1$-Calabi-Yau} and \emph{fractionally $\frac{2}{2}$-Calabi-Yau}.

In Section \ref{sec:symmetric} we study when the ribbon graph order $\A$ is symmetric.
These considerations apply also 
to certain finite-dimensional quotients of $\A$.
Such a quotient, which we will denote by $\AI$, can be obtained from a
Brauer graph algebra by replacing all commutativity relations with anti-commutative ones.

Using the first result we derive its analogue for the \emph{twisted Brauer graph algebra} $\AI$:

\begin{manualcorollary}{B}[Theorem \ref{thm:twisted-BGA} \ref{thm:twisted-BGA2}] \label{cor:2}
The 
functor $\nu_{\AI}$ is induced by some involution 
$\olnuu$
 of the algebra $\AI$.
So there are isomorphisms
$\mathbb{S}_\AI \cong (\mathstrut_{\text{\textemdash}})_{\scriptsize \olnuu}$ and $\mathbb{S}_\AI^2 \cong \Id$.
\end{manualcorollary}
In particular, the 
perfect 
derived 
category of $\AI$ 
 is \emph{twisted $0$-Calabi-Yau} and \emph{fractionally $\frac{0}{2}$-Calabi-Yau}.

Finally, we give several characterizations 
of symmetric ribbon graph orders:
\begin{manualtheorem}{C}[Theorem \ref{thm:symmetry}] \label{thm:3}
The following conditions are equivalent:
\begin{enumerate}
\item The order $\A$ is symmetric, that is, $\nu_\A \cong \Id$ or $\mathbb{S}_\A \cong [1]$,
\item The algebra $\AI$ is symmetric, that is, $\nu_\AI \cong \Id$ or $\mathbb{S}_{\AI} \cong \Id$, 
\item The {graph of }$\A$ is bipartite or the base field $\kk$ has characteristic two.
\end{enumerate}
\end{manualtheorem}

By this result,
twisted Brauer graph algebras
have very similar 
 Calabi-Yau properties as  
 preprojective algebras of Dynkin quivers \cite{Brenner-Butler-King}.
Moreover,
 we may associate several triangulated Calabi-Yau categories 
to any \emph{dimer model} or to any \emph{dessins d'enfants} (Remark \ref{rmk:dessin}).

To prove Theorem \ref{thm:1} we find a Frobenius form on the ribbon graph order.
It can be viewed as a lifted version of the Frobenius form of some Brauer graph algebra.
Corollary \ref{cor:2}
 follows from Theorem \ref{thm:1} and 
the observation 
that symmetricity ascends from orders to certain quotients
 (Proposition \ref{prp:bimodule-quotient}).\\
The main idea to
prove Theorem~\ref{thm:3} 
is to show that the ``minimal'' non-symmetric ribbon graph orders
arise from 
non-bipartite 
circuit graphs. On this way 
we clarify the relationship between ribbon graph orders and Brauer graph algebras.

In summary, any ribbon graph order 
is
``almost''
$1$-Calabi-Yau.
We hope that these results
have an analogue
in the symplectic geometry
of
a closed, oriented surface.
In a forthcoming paper \cite{Gnedin2}, we will investigate further homological properties of ribbon graph orders and twisted Brauer graph algebras.

\section*{Notation}
Throughout this paper we fix the following conventions:
\begin{itemize}[label=$\blackdiamond$]
\item
All modules will be finitely generated left modules.
\item The cardinality of a finite set $M$ will be denoted by $|M|$.
\item
Let $\kk$ denote an arbitrary field,
 $\Rx$  the ring of formal power series $\kk \llbracket t \rrbracket$ in one variable
and $\KK$ its quotient field $\kk(\!(t)\!)$.
\end{itemize}

\section{Basic notions on ribbon graph orders} \label{sec:basic}
In this section we introduce the main notions of this paper.\\
In more detail,  we define ribbon graph orders via complete gentle quivers in the first subsection.
To justify our terminology we 
associate a \emph{ribbon graph} to any complete gentle quiver
 in Subsection \ref{subsec:RGO-graphs}
and 
observe 
that any ribbon graph order is indeed an \emph{order} in the classical sense
in Subsection \ref{subsec:orders1}. 
In the last subsection, we 
recall notions related to the relative Serre functor for one-dimensional orders.

\subsection{From complete gentle quivers to ribbon graph orders}

The following class
of quivers 
 was introduced  by 
C.M.~Ringel \cite{Ringel}: 
\begin{dfn}\label{dfn:cg}
A finite quiver with relations $(Q,I)$ is called 
\emph{complete gentle}
if the following conditions are satisfied: 
\begin{center}
\begin{tabular}{p{0.75\textwidth}c}
\begin{enumerate}
\item \label{dfn:cg1}  There are exactly two arrows starting and exactly two arrows ending at each vertex. 
\item \label{dfn:cg2} For any arrow $b \in Q_1$ there 
are unique arrows $a,c \in Q_1$ such that $c \, b \notin I$ and $b \, a \notin I$.
\item \label{dfn:cg3} The ideal $I$ 
is
generated by paths of length two.
\end{enumerate} &
\quad
$
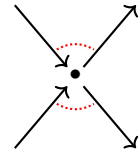
\begin{tikzcd}[
nodes={inner sep=1pt, outer sep=0pt},
baseline=1cm,
inner sep=0pt, outer sep=0pt,
every label/.append style={inner sep=1pt, font=\small, color=black},
column sep=0.6cm, 
cells={shape=circle},
arrows={->,thick,black},
ampersand replacement=\&
]
 \phantom{.} \ar[rd, ""{name=c, near end,
inner sep =-1pt}] \& \& \phantom{.} \\
  \& \bt \ar[rd, ""{name=d2, near start, inner sep=-1pt,, swap
}] \ar[ru, ""{name=b, near start, 
inner sep =-1pt
}]\&  \\
\phantom{.} 
\ar[ru,""{name=a, near end,swap, inner sep=-1pt
}]
 \& \& \phantom{.} 
\arrow[from=a, to=d2, densely dotted, red, bend right,-]
\arrow[from=c, to=b, densely dotted, red, bend left, -]
\end{tikzcd}$
\end{tabular}
\end{center}
\end{dfn}
The diagram on the right shows the ``local situation'' at each vertex in a complete gentle quiver, where
the zero relations are indicated by dotted lines.

The terminology above is motivated by the fact that any gentle quiver can be extended to a complete gentle quiver.
More precisely, 
the following holds:
\begin{rmk} 
A quiver with relations has a finite-dimensional \emph{gentle} path algebra
if and only if it
can be obtained by deleting one arrow from each oriented cycle of some complete gentle quiver.
\end{rmk}

The next series of examples comes from Lie theory:
\begin{ex} \label{ex:cuspidal}
For any $n \in \N^+$ let $(Q,I)_n$ be the following complete gentle quiver:
$$
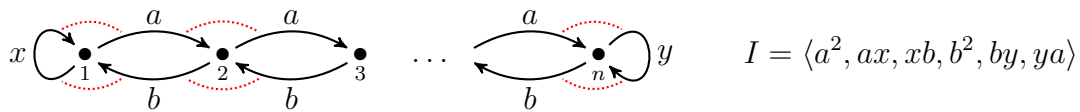
\begin{tikzcd}[nodes={
baseline=0pt,
inner sep=1pt}, 
every label/.append style={font=\normalsize, color=black},
column sep=1.45cm, cells={shape=circle},
arrows={->,thick, >=stealth'},
]
\bullet \ar[r, bend left, "a",
""{name=a1, near start, inner sep =-1pt},
""{name=a2, near end,  inner sep =-1pt}
] 
\arrow[out=135,in=225,loop, swap, "x", distance=1cm,<-,
""{name=x1, very near start, inner sep =-1pt},
""{name=x2, very near end, inner sep =-1pt}
]
& \bullet \ar[r, bend left, "a", 
""{name=c1, near start, inner sep =-1pt},
""{name=c2, near end, swap, inner sep =-1pt}
] 
 \ar[l, bend left, "b",
""{name=b1, near start, inner sep =-1pt},
""{name=b2, near end, inner sep =-1pt}
] 
& {\bullet}
  \ar[l, bend left, "b",
""{name=d1, near start, swap, inner sep =-1pt},
""{name=d2, near end, inner sep =-1pt}
]  
&[-1cm] \ldots
&[-1.5cm]
 \phantom{\bullet}   \ar[r, bend left, "a", 
""{name=e1, near start, inner sep =-1pt},
""{name=e2, near end,  inner sep =-1pt}
] 
&
 \bullet  \ar[l, bend left, "b",
""{name=f1, near start,  inner sep =-1pt},
""{name=f2, near end, inner sep =-1pt}
]  
\arrow[out=45,in=315,loop, "y", distance=1cm,
""{name=z1, very near start, inner sep =-1pt},
""{name=z2, very near end, inner sep =-1pt}
] \\[-1.05cm]
{\scriptstyle 1} & {\scriptstyle 2} & {\scriptstyle 3} & & & {\scriptstyle n}
\arrow[from=x1, to=a1, densely dotted, red, bend left,-]
\arrow[from=a2, to=c1, densely dotted, red, bend left,-]
\arrow[from=d2, to=b1, densely dotted, red, bend left,-]
\arrow[from=b2, to=x2, densely dotted, red, bend left,-]
\arrow[from=e2, to=z1, densely dotted, red, bend left,-]
\arrow[from=z2, to=f1, densely dotted, red, bend left,-]
\end{tikzcd} 
\qquad
I = \langle
a^2, ax, xb, b^2, by, ya \rangle
$$
The quiver $(Q,I)_n$ 
appears in the study of
 cuspidal representations over the Lie algebra $\mathfrak{sl}_{n+1}(\C)$ by work of Grantcharov and Serganova \cite{Grantcharov-Serganova}.
\end{ex}

Throughout this paper, we will use the following notation:
\begin{notation}\label{not:arrow}
Let $a \in Q_1$ be an arrow in a complete gentle quiver $(Q,I)$.
\begin{itemize}[label={$\blackdiamond$}]
\item 
We denote
by $\sigma(a)$ the unique arrow such that $\sigma(a) \,a \notin I$. This defines a permutation $\sigma\!: Q_1 \to Q_1$ which we will call \emph{the permutation of the quiver $(Q,I)$}.
\item We denote 
by $c_a$ the unique repetition-free cyclic path starting with the arrow $a$ and
 by $n(a)$ the length
 of the cyclic path $c_a$.
\end{itemize}
\end{notation}

\begin{ex}\label{ex:3quivers}
We will use the following examples of complete gentle quivers:
$$
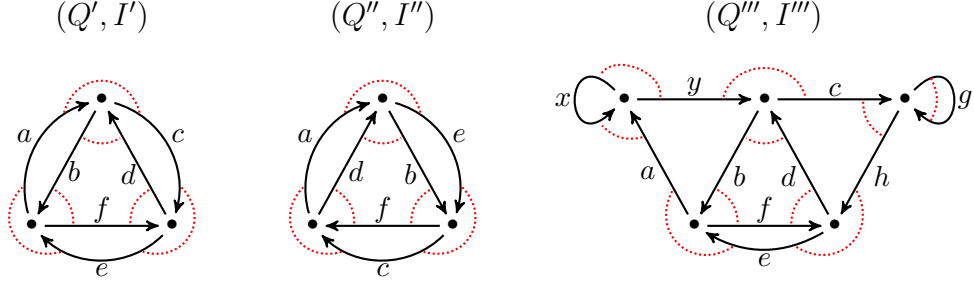
\begin{array}{ccc}
(Q',I') \quad& (Q'',I'')\quad & (Q''',I''')\\
\begin{tikzcd}[,nodes={
baseline=0pt,
inner sep=1pt}, 
every label/.append style={inner sep=1pt, font=\small, color=black},
column sep=0.6cm, cells={shape=circle},
row sep=1.35cm,
arrows={->,thick,black,
 >=stealth',yshift=0mm},
ampersand replacement=\&
]
\& 
\bt 
\arrow[rd,bend left,looseness=1, "{c}",
""{name=e1, pos=0.2, description,transparent, inner sep =-1pt},
""{name=e2,  pos=0.8,description,transparent, inner sep =-1pt}, out=45, in = 135
]
\ar[rd, "d"{xshift=1pt, yshift=1pt}, swap, <-,
""{name=a1, pos=0.25,description,transparent, inner sep =-1pt},
""{name=a2, pos=0.75,description,transparent, inner sep =-1pt} 
]
\&
\\
\ar[ru, "{b}"{xshift=-1pt, yshift=2pt}, <-, swap,
""{name=c1, pos=0.25,description,transparent, inner sep =-1pt},
""{name=c2, pos=0.75,description,transparent, inner sep =-1pt}]
\bt \ar[ru, bend left,looseness=1, "{a}", 
""{name=d1, pos=0.2,description,transparent, inner sep =-1pt},
""{name=d2, pos=0.8,description,transparent, inner sep =-1pt}, out=45, in = 135
]  
\& \&
\bt 
\ar[ll, "f", swap,<-,
""{name=b1, pos=0.25, description,transparent,inner sep =-1pt},
""{name=b2, pos=0.75,description,transparent, inner sep =-1pt}] 
\ar[ll, bend left, looseness=1, "{e}",
""{name=f1, pos=0.2,description,transparent, inner sep =-1pt},
""{name=f2, pos=0.8, description,transparent,inner sep =-1pt} , out=45, in = 135
]
\arrow[from=e2, to=f1, densely dotted, red, bend left=75,  looseness = 1.5, -]
\arrow[from=f2, to=d1, densely dotted, red, bend left=75,  looseness = 1.5, -]
\arrow[from=d2, to=e1, densely dotted, red, bend left=75,  looseness = 1.5, -]
\arrow[from=c2, to=a1, densely dotted, red, bend right, -]
\arrow[from=a2, to=b1, densely dotted, red, bend right, -]
\arrow[from=b2, to=c1, densely dotted, red, bend right, -]
\end{tikzcd}
\quad
&
\begin{tikzcd}[,nodes={
baseline=0pt,
inner sep=1pt}, 
every label/.append style={inner sep=1pt, font=\small, color=black},
column sep=0.6cm, cells={shape=circle},
row sep=1.35cm,
arrows={->,thick,black,
 >=stealth',yshift=0mm},
ampersand replacement=\&
]
\& 
\bt 
\arrow[rd,bend left,looseness=1, "{e}",
""{name=e1, pos=0.2, description,transparent, inner sep =-1pt},
""{name=e2,  pos=0.8,description,transparent, inner sep =-1pt}, out=45, in = 135
]
\ar[rd, "b"{xshift=1pt, yshift=1pt}, swap,
""{name=a1, pos=0.25,description,transparent, inner sep =-1pt},
""{name=a2, pos=0.75,description,transparent, inner sep =-1pt} 
]
\&
\\
\ar[ru, "{d}"{xshift=-1pt, yshift=2pt}, , swap,
""{name=c1, pos=0.25,description,transparent, inner sep =-1pt},
""{name=c2, pos=0.75,description,transparent, inner sep =-1pt}]
\bt \ar[ru, bend left,looseness=1, "{a}", 
""{name=d1, pos=0.2,description,transparent, inner sep =-1pt},
""{name=d2, pos=0.8,description,transparent, inner sep =-1pt}, out=45, in = 135
]  
\& \&
\bt 
\ar[ll, "f", swap,
""{name=b1, pos=0.25, description,transparent,inner sep =-1pt},
""{name=b2, pos=0.75,description,transparent, inner sep =-1pt}] 
\ar[ll, bend left, looseness=1, "{c}",
""{name=f1, pos=0.2,description,transparent, inner sep =-1pt},
""{name=f2, pos=0.8, description,transparent,inner sep =-1pt} , out=45, in = 135
]
\arrow[from=e2, to=f1, densely dotted, red, bend left=75,  looseness = 1.5, -]
\arrow[from=f2, to=d1, densely dotted, red, bend left=75,  looseness = 1.5, -]
\arrow[from=d2, to=e1, densely dotted, red, bend left=75,  looseness = 1.5, -]
\arrow[from=c2, to=a1, densely dotted, red, bend right, -]
\arrow[from=a2, to=b1, densely dotted, red, bend right, -]
\arrow[from=b2, to=c1, densely dotted, red, bend right, -]
\end{tikzcd}
\quad
&
\begin{tikzcd}[,nodes={
baseline=0pt,
inner sep=1pt}, 
every label/.append style={inner sep=1pt, font=\small, color=black},
column sep=0.6cm, cells={shape=circle},
row sep=1.35cm,
arrows={->,thick,black,
 >=stealth',yshift=0mm},
ampersand replacement=\&
]
\bt
\arrow[out=135,in=225,loop, swap, "x", distance=1cm,
""{name=x1, very near start, inner sep =-1pt},
""{name=x2, very near end, inner sep =-1pt}
]
\ar[rr, "y",
""{name=a1, pos=0.2, description,transparent,inner sep =-1pt},
""{name=a2, pos=0.75,description,transparent, inner sep =-1pt}]
\& \&
\bt 
\ar[rr, "{c}"{xshift=0pt, yshift=0pt},
""{name=d1, pos=0.25,description,transparent, inner sep =-1pt},
""{name=d2, pos=0.75,description,transparent, inner sep =-1pt}]
\ar[ld, "{b}"{xshift=0pt, yshift=0pt}, 
""{name=b1, pos=0.25,description,transparent, inner sep =-1pt},
""{name=b2, pos=0.75,description,transparent, inner sep =-1pt}]
\& \& \bt
\ar[ld, "h", 
""{name=f1, pos=0.25, description,transparent,inner sep =-1pt},
""{name=f2, pos=0.75,description,transparent, inner sep =-1pt}]
\arrow[out=45,in=315,loop, "g", distance=1cm,
""{name=e1, very near start, inner sep =-1pt},
""{name=e2, very near end, inner sep =-1pt}
]
\\
\& 
\bt 
\ar[rr,bend left=0, "f",
""{name=i1, pos=0.25,description,transparent, inner sep =-1pt},
""{name=i2, pos=0.75,description,transparent, inner sep =-1pt}]
\ar[lu, "{a}"{xshift=0pt, yshift=0pt}, 
""{name=c1, pos=0.25,description,transparent, inner sep =-1pt},
""{name=c2, pos=0.75,description,transparent, inner sep =-1pt}]
\&
\& \bt
\ar[ll, "e", bend left,
""{name=h1, pos=0.25,description,transparent, inner sep =-1pt},
""{name=h2, pos=0.75,description,transparent, inner sep =-1pt}]
\ar[lu, "d"{xshift=0pt, yshift=0pt}, 
""{name=g1, pos=0.25,description,transparent, inner sep =-1pt},
""{name=g2, pos=0.75,description,transparent, inner sep =-1pt} 
]
\arrow[from=x1, to=a1, densely dotted, red, bend left=75, -]
\arrow[from=a2, to=d1, densely dotted, red,
bend left=75, looseness=1.25,
 -]
\arrow[from=d2, to=f1, densely dotted, red, bend right, -]
\arrow[from=f2, to=h1, densely dotted, red, bend left=75,looseness=1.5, -]
\arrow[from=h2, to=c1, densely dotted, red, bend left=75,looseness=1.5, -]
\arrow[from=c2, to=x2, densely dotted, red, bend left, -]
\arrow[from=g2, to=b1, densely dotted, red, bend left, -]
\arrow[from=b2, to=i1, densely dotted, red, bend left, -]
\arrow[from=i2, to=g1, densely dotted, red, bend left, -]
\arrow[from=e2, to=e1, densely dotted, red, bend right, -]
\end{tikzcd}
\end{array}
 $$
The quiver $(Q',I')$ 
 appears in cluster theory.
It has permutation  
 $\sigma = (ba)(dc)(fe)$ and six repetition-free cycles, for example, $c_a = ba$ and $c_b = ab$.
Moreover, the quiver $Q'$ has a potential 
$W = eca + cae + aec+fdb+dbf + bfd$
such that
$I = \langle
ca, ae, ec, db, bf, fd
\rangle = \langle \partial_a W \ | \ a\in Q_1 \rangle$.
\end{ex}

Since each
arrow of a complete gentle quiver 
$(Q,I)$
 lies on an oriented cycle, its path algebra $\kk Q/I$ is infinite-dimensional.
Ribbon graph orders
 are  given essentially by such path algebras:

\begin{dfn}\label{dfn:RGO}
A 
ring
$\A$ will be called a \emph{ribbon graph order} if there is
 a field $\kk$ and
a complete gentle quiver $(Q,I)$ such that
the arrow ideal completion of the path algebra $\kk Q/I$
is Morita equivalent
to the ring
$\A$.
 \end{dfn}

\begin{ex}
The nodal curve singularity $\A = \kk \llbracket x,y \rrbracket / (xy)$ is 
 isomorphic to
the arrow ideal completion of the path algebra $\kk [ x,y ] /(xy)$ of a complete gentle quiver.
\end{ex}

\subsection{The ribbon graph of a complete gentle quiver}\label{subsec:RGO-graphs}

We will use the next definition to characterize  symmetric ribbon graph orders:
\begin{constr}\label{constr:graph}
For any complete gentle quiver $(Q,I)$
we define
its \emph{graph} $\UG:$
\begin{enumerate}
\item 
The \emph{set of nodes} $\UG_0$ of the graph $\UG$ is given by the set of $\sigma$-orbits in $Q_1$.
For any arrow $a \in Q_1$ we denote by $v_a = \{\, \sigma^p(a) \ | \ 
1 \leq
 p \leq n(a) 
 \} $
  its $\sigma$-orbit.
\item
For each vertex $i \in Q_0$ 
we set up an \emph{edge} $e_i$
between the vertices $v_a$ and $v_b$ where $a$ and $b$ denote the two distinct arrows from $Q_1$ starting at the vertex $i$.\\
Such an edge is a \emph{loop}
if $v_a = v_{b}$, otherwise it is called 
\emph{ordinary}.
\end{enumerate}
In particular, the graph $\UG$ may have multiple edges or loops.
\end{constr}
\begin{ex}\label{ex:line}
\begin{enumerate}
\item 
The graph of the ribbon graph order $\kk \langle \!\langle a,b \rangle\! \rangle/(a^2,b^2)$ is a loop.
\item
The graph of the quiver $(Q,I)_n$ from Example \ref{ex:cuspidal} is a line with $n$ edges:
$$\begin{rg} \phantom{.} \ar[-]{r} \& \phantom{.} \ar[-]{r} \& \phantom{.} \ar[draw=white,-]{rr}[description]{\ldots} \& \&    \phantom{.} \ar[-]{r} \&  \phantom{.} \end{rg}$$
\end{enumerate}
\end{ex}
To motivate our terminology, we include a further construction:
\begin{constr} Let $(Q,I)$ be a complete gentle quiver, $\UG$ its graph and $v \in \UG_0$.
The edges incident to the node $v$ are given by the set 
$E_v 
= \{ \, e_{s(a)} \ | \ a \in Q_1\!: v_a = v  \, \}
$. The map 
$
\begin{td} \sigma_v\!: E_v \ar{r} \&  E_v\end{td}$, $\begin{td} e_{s(a)} \ar[mapsto]{r} \& e_{t(a)} \end{td}$
is a cyclic permutation on the set $E_v$. 
The \emph{ribbon graph} of the quiver $(Q,I)$ is given by 
the datum $(\UG, (\sigma_v)_{v\in \UG_0})$.
\end{constr}
\begin{ex}
The ribbon graph of the quiver $(Q'',I'')$ from Example \ref{ex:3quivers} is given by three loops at one vertex $v$ and the 
permutation $\sigma_v = 
(e_{s(a)} e_{s(b)} e_{s(c)} e_{s(d)} e_{s(e)} e_{s(f)})$.
The permutation $\sigma_v$ can be read off from the 
 permutation $\sigma = (abcdef)$ of $(Q'',I'')$.
This is also true in general.
\end{ex}
There is a natural relationship
between idempotent subalgebras of a ribbon graph order and subgraphs of its  graph:
\begin{lem}\label{lem:subgraph}
Let $\A$ be a ribbon graph order, $(Q,I)$ a complete gentle quiver of $\A$ and $\UG$ its graph.
Let $Q'_0$ be some proper subset of
$Q_0$, $e = \sum_{i \in Q'_0} e_i$, and $\A'  = e \A e$.\\
Then 
the idempotent subring $\A'$ 
is a ribbon graph order.
Moreover, its graph $\UG'$
can be obtained from the graph $\UG$ by
removing all edges $e_i$ 
with index $i \in Q_0 \backslash Q'_0$ 
and deleting all nodes of valency zero.
\end{lem}
\begin{proof}
The 
ring
$\A'$ is Morita equivalent to the arrow ideal completion of the path algebra of some quiver $(Q',I')$.
This quiver can be described as follows:
\\
Its vertices are given by $Q'_0$
and its set of arrows by
$Q'_1 = \{ a \in Q_1 \ | \ s(a) \in Q'_0 \}$.
For any arrow $a \in Q'_1$ the number $m = \min\{1 \leq p \leq n(a) \ | \ \sigma^p(a) \in Q'_1\}$ is well-defined.
This yields a map $\sigma': Q'_1 \to Q'_1, a \mapsto \sigma^m(a)$.
The start and target maps $s'$ and $t': Q'_1 \to Q'_0$
are given by $s' = s$ respectively $a \mapsto t'(a) = s(\sigma'(a))$.
The ideal of relations is $I' = \langle b\,a \ | \ a,b \in Q'_1: s'(b) = t'(a), b \neq \sigma'(a) \rangle$.\\
It follows that the quiver
$(Q',I')$ is complete gentle 
and 
the map $\sigma'$ is actually its permutation in the sense of Notation \ref{not:arrow}.
In particular, $\A'$ is a ribbon graph order.\\
With this quiver presentation  of the ribbon graph order $\A'$ 
the  description of its graph $\UG'$ follows directly from Construction \ref{constr:graph}. \qedhere
\end{proof}
\begin{ex}
The quivers $(Q',I')$ and $(Q''',I''')$ of Example \ref{ex:3quivers} have the graphs
$$
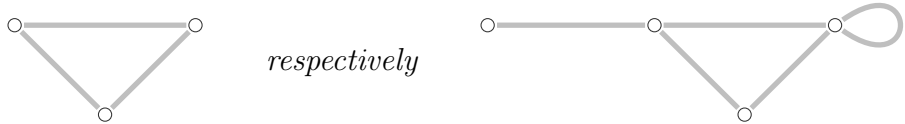
\begin{rg}
\phantom{.}
\ar[-]{rr} \ar[-]{rd} \& \& \phantom{.}
\\
\& 
\phantom{.}  \ar[-]{ru} 
\&
\end{rg}
\qquad \text{respectively} \qquad
\begin{rg}
\phantom{.} \& \& \phantom{.} \ar[-]{ll} \ar[-]{rr} \ar[-]{rd} \& \& \phantom{.} \arrow[-,out=35,in=325,loop, 
distance=1.25cm, 
] \\
\& \& \& \phantom{.} \ar[-]{ru} \&
\end{rg}
$$
In particular, the completed path algebra
of the quiver $(Q',I')$ is isomorphic to an idempotent subring of the completed path algebra of the quiver $(Q''',I''')$.
\end{ex}

\subsection{Ribbon graph orders are Bäckström orders}
\label{subsec:orders1}

In this subsection, we 
show that
ribbon graph orders are one-dimensional orders.
\begin{dfn}\label{dfn:order}
As above, let $\Rx = \kk \llbracket t \rrbracket$.
An $\Rx$-algebra $\A$ will be called an $\Rx$-order, or simply an \emph{order}, if
the $\Rx$-module  $\mathstrut_{\Rx} \A$ is a free module of finite rank.
\end{dfn}
\begin{rmk}\label{rmk:algebra-structure}
By \cite{Goto-Nishida}*{Corollary 3.2}
an $\Rx$-algebra $\A$ is an $\Rx$-order if and only if 
$\Hom_{\A}(\A/\rad\A,\A) = 0$. 
In particular, the $\Rx$-order property is independent from the choice of an $\Rx$-algebra morphism $\Rx \to \A$.
\end{rmk}
We work over 
the
complete local
 ring $\Rx$
instead of the polynomial ring $\kk [t]$ for the following technical reason:
\begin{rmk} 
Let $\A$ be an $\Rx$-algebra such that 
the module
$\mathstrut_\Rx \A$ is finitely generated.\\
Then
 $\A$ is a \emph{semiperfect} ring, that is, 
any simple $\A$-module has a projective cover. 
In particular, any indecomposable projective $\A$-module is isomorphic to $\A e$ for some primitive idempotent $e$ of $\A$.
Moreover,  
the $\Rx$-algebra $\A$ is Morita equivalent to some basic $\Rx$-algebra $e \A e$ for some idempotent $e \in \A$.
\end{rmk}

To show that any ribbon graph order $\A$ is an order in the sense above, 
we will embed the ring $\A$ in a hereditary order $\Gamma$. Before doing that, we recall a few notions.
An order $\Gamma$ is \emph{hereditary} if any simple $\Gamma$-module has projective dimension one.
Cyclic quivers yield examples of such orders:
\begin{ex}\label{ex:hereditary}
Let $\wt{Q}$ be the $n$-cycle quiver and
$\Gamma$ the arrow ideal completion of its path algebra $\kk \wt{Q}$.
Let $\mx = (t)$ denote the maximal ideal of the local ring
$\Rx$.
\begin{enumerate}
\item If $n=2$, then $\Gamma$ is isomorphic to the following subring of $\Mat_{2 \times 2}(\Rx):$
$$
\wt{Q} = \begin{tikzcd}[ampersand replacement=\&,
nodes={ inner sep=1pt}, labels={rectangle, font=\normalsize, inner sep=1pt, outer sep=1pt, minimum size=0cm},]
\underset{1}{\bullet}\ar[bend left]{r}{a} \& \underset{2}{\bullet} \ar[bend left]{l}{b}
\end{tikzcd} \qquad
\Gamma 
\overset{\sim}{\longrightarrow}
\begin{bmatrix} \Rx & \mx \\ \Rx & \Rx \end{bmatrix}, \quad
e_1 \mapsto \begin{bmatrix}
1 & 0 \\
0 & 0
\end{bmatrix},\
a \mapsto \begin{bmatrix}
0 & 0 \\
1 & 0
\end{bmatrix},  \
b \mapsto \begin{bmatrix}
0 & t \\
0 & 0
\end{bmatrix}.
$$
\item For each $n \in \N^+$, $\Gamma$ is isomorphic to the following 
``triangular'' 
matrix ring:
$$
\begin{td}
\Gamma \ar{r}{\sim} \&
\mathrm{T}_n(\Rx)
 := 
\{
A 
\in \Mat_{n \times n}(\Rx) \ | \ 
a_{ij} \in \mx \text{ for any two indices }1 \leq  i < j \leq n\},
\end{td}
$$
\end{enumerate}
In particular, the center of $\Gamma$ is isomorphic to $\Rx$.
As for complete gentle quivers, for any arrow $a \in \wt{Q}_1$ there is a unique minimal cycle $c_a$ in $\wt{Q}$
starting at $a$.\\
Let $m \in \N^+$ and
$\phi_m\!: \Rx \to \Gamma, \ t \mapsto \sum_{a\in \wt{Q}_1} c_a^m$, that is, 
$\phi_m(t)$ is given by 
the diagonal matrix
$t^m \cdot \mathbbm{1}_n$.
Then the ring $\Gamma$ is a hereditary $\Rx$-order and 
 $\rk_\Rx \Gamma = \ell(\phi_m(t)) = m \, n^2$. 
\end{ex}
Any complete gentle quiver gives rise to some cyclic quivers:
\begin{constr}
\label{constr:normalization}
For any complete gentle quiver $(Q,I)$ 
its \emph{normalization} $\wt{Q}$ is a quiver defined as follows:
\begin{enumerate}
\item 
Its vertices are given by $\wt{Q}_0 = \{ i_a \ | \ a \in Q_1 \}$ and its arrows by $\wt{Q}_1 = Q_1$. 
\item 
The start and the target map $\wt{s}, \wt{t}\!: \wt{Q}_1 \to \wt{Q}_0$ 
are given by $\wt{s}(a) = i_a$ respectively
$\wt{t}(a) = i_{\sigma(a)}$ for each arrow $a \in \wt{Q}_1$.
\end{enumerate}
In particular,
the quiver
$\wt{Q}$ has the following description:\\
Let 
$(n_1, \ldots, n_p)$
denote the cycle type of the
 permutation $\sigma$ of $(Q,I)$
and
 $\wt{Q}_j$ the cyclic quiver with $n_j$ arrows
for each index $1 \leq j \leq p$.
Then 
 $\wt{Q} = \prod_{j=1}^p \wt{Q}_j$.
\end{constr}
\begin{ex}
The normalization of the quiver $(Q''',I''')$ from Example \ref{ex:3quivers}
is given by the product of 
a loop, a two-, a three- and a four-cycle-quiver.
\end{ex}
The following class of rings was studied by Ringel and Roggenkamp \cite{Ringel-Roggenkamp}:
\begin{dfn}\label{dfn:baeckstroem}
 An $\Rx$-algebra $\A$ is \emph{Bäckström}
if there is some hereditary $\Rx$-order $\Gamma$ and a monomorphism of $\Rx$-algebras $\iota\!: \A \hookrightarrow \Gamma$ such that $\iota(\rad \A) = \rad \Gamma$.
\end{dfn}
\begin{lem}\label{lem:baeckstroem1}
Let $\A$ be a Bäckström $\Rx$-algebra with some hereditary overorder $\Gamma$.
Then 
$\A$ is an $\Rx$-order 
such that
$\rk_{\Rx} \A = \rk_{\Rx} \Gamma$ and $\KK \otimes_{\Rx} \A \cong \KK \otimes_{\Rx} \Gamma$.
\end{lem}
\begin{proof}
Note that the cokernel of the monomorphism $\iota\!: \A \hookrightarrow \Gamma$ has finite length.
This implies the three claims.
\end{proof}
We note that there is a generalization of Lemma \ref{lem:baeckstroem1} given by \cite{Burban-Drozd18}*{Lemma 2.7}.

The next lemma collects the basic properties of ribbon graph orders for later use:
\begin{prp}\label{prp:RGO}
Let $(Q,I)$ be any complete gentle quiver, 
$\wt{Q}$ its normalization quiver.
Let
$\A$ and $\Gamma$ be
 the arrow ideal completions of the path algebras $\kk Q/I$ respectively $\kk \wt{Q}$.
Let $m\!: \UG_0 \to \N^+$ be any map on the set $\UG_0$ of $\sigma$-orbits of $Q_1$.
We set
$$
\begin{td} \phi_m\!: \Rx= \kk \llbracket t \rrbracket  \ar{r}\& \A \end{td}, \qquad \begin{td}t \ar[mapsto]{r} \&\sum_{a \in Q_1} c_a^{m(v_a)} \end{td}.
$$
Then 
the ring $\A$ is an $\Rx$-order.
 More precisely, the following statements hold:
\begin{enumerate}
\item   \label{prp:RGO1}
 The ring $\A$ is a Bäckström $\Rx$-order with hereditary overorder $\Gamma$.
\item  \label{prp:RGO2} 
It holds that
$\rk_{\Rx} \A = 
\ell(\phi_m(t)) = 
{\sum}_{a \in Q_1} m(v_a) \, n(a)$.
\item \label{prp:RGO3} 
The
$\KK$-algebra
$\KK \otimes_{\Rx} \A$ 
is Morita equivalent to the product $\KK^{\times p}$,
where $p= |\UG_0|$.
\item \label{prp:RGO4} For any arrow $a \in Q_1$ the minimal projective resolution of $\A a$ is periodic.
\end{enumerate}
\end{prp}
\begin{proof}
\begin{enumerate}
\item
It is straightforward to check that $\phi_m(t)$ is a central element of $\A$.
Thus $\phi_m$ defines an $\Rx$-algebra structure on $\A$.\\
Let $\{e_i \ | \ i \in Q_0\}$ and $\{{\wt{e}}_{i_a} \ | \ a\in Q_1 \}$ denote complete sets of 
primitive orthogonal idempotents of the semiperfect ring $\A$ respectively $\Gamma$.
Let $\begin{td} \iota\!: \A \ar[hookrightarrow]{r}\& \Gamma \end{td}$ 
be the unique morphism of $\kk$-algebras such that
$a \mapsto a$ for any $a \in Q_1$ and
$e_i \mapsto \sum_{a \in Q_1: s(a) = i}{\wt{e}}_{i_a}$ for any $i \in Q_0$.
Then $\iota$ is a monomorphism
of $\Rx$-algebras such that $\iota(\rad \A) = \rad \Gamma$.
Thus $\A$ is a Bäckström $\Rx$-order with hereditary overorder $\Gamma$
by  Lemma \ref{lem:baeckstroem1}.
\end{enumerate}
Let $v_{a_1}, \ldots, v_{a_p}$ be a 
complete set of representatives of the set of $\sigma$-orbits $\UG_0$. 
Denote $m_j = m(v_{a_j})$ and $n_j = n(a_j)$ 
for any index $1 \leq j \leq p$.
 Construction \ref{constr:normalization} 
yields that $\Gamma \cong \prod_{j=1}^p \mathrm{T}_{n_j}(\Rx)$
in the notation of
 Example \ref{ex:hereditary}.
The last two statements in Lemma \ref{lem:baeckstroem1} yields the remaining two claims:
\begin{enumerate}
\setcounter{enumi}{1}
\item It holds that
$ \rk_{\Rx} \A = \rk_{\Rx} \Gamma = \ell(\phi_m(t)) =
\sum_{j=1}^p m_j \, n_j^2
= \sum_{a \in Q_1} m(v_a) n(a)$.
\item There are $\KK$-algebra isomorphisms $\KK \otimes_{\Rx} \A \cong \KK \otimes_{\Rx} \Gamma \cong \prod_{j=1}^p \Mat_{n_j \times n_j} (\KK)$. 
\item For any arrow $a \in Q_1$
there is a unique arrow $\phi(a) \in Q_1$ such that $s(\phi(a)) = t(a)$ and $\phi(a) \, a \in I$ by Definition \ref{dfn:cg}.
Then the minimal projective resolution of $\A a$ has period $k = \min \{p \in \N^+ \ | \ \phi^p(a) = a\}$.
\qedhere
\end{enumerate}
\end{proof}

The last proposition shows that any ribbon graph order is Morita equivalent to an $\Rx$-order, and thus also an $\Rx$-order. 

 Proposition \ref{prp:RGO} \ref{prp:RGO3} shows that for any  ribbon graph order $\A$ the  $\KK$-algebra $\KK \otimes_{\Rx} \A$ is \emph{split semisimple}.
In summary, any ribbon graph order is a classical order over the complete discrete valuation ring $\Rx$.

\begin{rmk}
Kauer and Roggenkamp studied generalized versions of ribbon graph orders over a regular local ring $S$ of arbitrary Krull dimension \cite{Kauer-Roggenkamp}.
Our specialization to 
the ring
$\Rx$ 
is motivated by the relationship of  ribbon graph orders to gentle quivers. 
\end{rmk}

\subsection{Relative Serre duality for orders
}
Throughout this subsection, 
let $\A$ 
denote any $\Rx$-order.
In Section \ref{sec:canonical} we will study the following  notions specialized to ribbon graph orders:
\begin{dfn}
\begin{enumerate}
\item 
The \emph{canonical
bimodule of }$\A$
is given by 
$\omegac = \Hom_{\Rx}(\A, \Rx)$.
\item 
The 
\emph{derived Nakayama functor}
of $\A$
is the left-derived tensor product
$$
\begin{td} \nu = \omegac \, {\overset{\mathbb{L}}{\otimes}}_{\A}
\mathstrut_{\text{\textemdash}} :
 \D^-(\A) \ar{r} \&  \D^-(\A) \end{td}$$
on the right-bounded derived category of finitely generated $\A$-modules.
\end{enumerate}

\end{dfn}

It is well-known that 
the category $\D^-(\A)$ 
has 
 a \emph{relative Serre functor}
$$\mathbb{S} = \nu \circ [1]: 
\begin{td} \D^-(\A) \ar{r} \&  \D^-(\A) \end{td}$$
given by the 
by the composition of the functor $\nu$ with the shift functor.
 In different setups this result
was obtained by van den Bergh \cite{van_den_Bergh}*{Lemma 6.4.1}, Ginzburg \cite{Ginzburg}*{Theorem 7.2.14}, Keller \cite{Keller}*{Lemma 4.1}
and Iyama and Reiten \cite{Iyama-Reiten}*{Theorem 3.7}. Our exposition follows mainly the last reference:

\begin{thm}
\label{thm:vdB/IR}
In the setup above, 
the following statements hold:
\begin{enumerate}
\item 
For any perfect complex $\CP \in \Perf(\A)$
and any
bounded complex $\CM \in \Dbfd(\A)$
with finite-dimensional  cohomology 
at each degree
there is an isomorphism
\begin{align} 
\label{eq:serre-duality}
\tag{$\star$}
\Hom_{\D(\A)}(\CP,\CM) \cong 
{\Hom_{\D(\A)}(\CM,\mathbb{S}(\CP))}^*
\end{align}
which is
 functorial in $\CP$ as well as $\CM$, where
 ${(\mathstrut_{\text{\textemdash}})}^*$ 
denotes the $\kk$-linear duality.
\item If the order $\A$ is \emph{Gorenstein},
that is,  $\omega$ is projective as left $\A$-module,
 then
the functor $\mathbb{S}$ restricts to an auto-equivalence of the full subcategory $\Perffd(\A)$ in $\Perf(\A)$
of perfect complexes with finite-dimensional cohomology at each degree.
\end{enumerate}
\end{thm}
The assumptions on the complexes in the theorem above
ensure that the morphism spaces in \eqref{eq:serre-duality} have finite-dimension as $\kk$-linear spaces.
Let us also note that $\Perffd(\A)$ is  a $\kk$-linear triangulated and Hom-finite category. 
It was already known to Kauer and Roggenkamp that any ribbon graph order is Gorenstein \cite{Kauer-Roggenkamp}.

The derived Nakayama functor defines the derived Auslander-Reiten translation of any complex from the category $\Perffd(\A)$: 
\begin{cor}[Happel]\label{cor:AR-theory}
For any indecomposable complex $\CP \in \Perffd(\A)$
there is an \emph{Auslander-Reiten triangle}
 of the form
$$\begin{td} \nu(\CP) \ar{r} \& E^{\bt} \ar{r} \& \CP \ar{r} \& \mathbb{S}(\CP)
\end{td} \text{ in $\D^-(\A)$.}$$
\end{cor}
\begin{proof}
The statement
follows from Theorem \ref{thm:vdB/IR}
by a straightforward adaptation of an argument due to Happel \cite{Happel}*{Proof of Theorem 4.5}. 
\end{proof}

This fact yields another motivation for the study of the derived Nakayama functor for ribbon graph orders in the next section.

In Section \ref{sec:symmetric} we will study when ribbon graph orders belong to the following class:

\begin{thm}[\cite{Iyama-Reiten}*{Theorem 3.8}]
The following conditions are equivalent:
\begin{enumerate}
\item 
The $\Rx$-order $\A$
is \emph{symmetric}, that is, there is an 
 isomorphism $\A \cong \omegac$ of $\A$-bimodules.
\item \label{eq:1-CY}
For any perfect complex $\CP \in \Perf(\A)$
and any
 complex $\CM \in \Dbfd(\A)$
with finite-dimensional total cohomology 
there is a bifunctorial isomorphism
\begin{align} 
\label{eq:serre-duality2}
\tag{$\star\star$}
\Hom_{\D(\A)}(\CP,\CM) \cong 
{\Hom_{\D(\A)}(\CM,\CP[1])}^*
\end{align}
\end{enumerate}
In this case,
the category 
$\Perffd(\A)$ is \emph{$1$-Calabi-Yau}, that is, there
is a {natural} isomorphism $\mathbb{S}(\CP) \cong \CP[1]$ for any complex $\CP \in \Perffd(\A)$, where $\mathbb{S}
=  \nu \circ [1]$.
\end{thm}
Since property \ref{eq:1-CY} above does not refer to the base ring $\Rx$ the following holds:
\begin{cor}\label{cor:symmetry}
The property of a ring $\A$ to be a \emph{symmetric $\Rx$-order} is independent from the choice of an $\Rx$-algebra morphism $\Rx \to \A$.
\end{cor}

\begin{rmk}
By Proposition \ref{prp:RGO} \ref{prp:RGO4}
 any ribbon graph order $\A$
has infinite global dimension. In particular, the category $\Perf(\A)$
is a proper subcategory of the bounded derived category $\Db(\A)$. 
\end{rmk}

\section{The canonical bimodule of a ribbon graph order}\label{sec:canonical}
To simplify the exposition we assume for the rest of the paper that
\begin{itemize}[label={$\blackdiamond$}]
\item any $\Rx$-order is a basic $\Rx$-algebra.
\end{itemize}
Throughout this section, we fix the following setup:
\begin{itemize}[label=$\blackdiamond$]
\item Let $\A$ be a ribbon graph order, $\kk$ its base field and $(Q,I)$ a complete gentle quiver of $\A$.  
Motivated by Corollary \ref{cor:symmetry}
we choose the
$\Rx$-algebra morphism 
$\Rx  = \kk \llbracket t \rrbracket \to \A, \, t \mapsto \sum_{a \in Q_1} c_a$. We set
$\A^\vee= \Hom_\Rx(\A, \Rx)$.
\end{itemize}
To study
the 
derived Nakayama functor of the order $\A$, 
we give
 an
explicit description of its canonical bimodule $\A^\vee$
in the first subsection.
It turns out that the bimodule $\A^\vee$ is isomorphic to the bimodule $\A_\nuu$ with regular left and 
 right module structure
 twisted by some involution $\nuu$ of $\A$.

In the second subsection we derive our description of the canonical bimodule
using
 the
 following approach:
\begin{enumerate}
\item We choose  $\Rx$-linear bases $\mathcal{B}$ and $\mathcal{B}^\vee$  
of the $\Rx$-modules $\mathstrut_{\Rx} \A$ respectively $\mathstrut_{\Rx} \omegac$.
\item We define an $\Rx$-linear form $\phi$
which induces a bimodule isomorphism $\A_{\nuu} \cong \A^{\vee}$.
\end{enumerate}
This method is 
well-known in the setup of finite-dimensional Frobenius algebras
and based on the Brauer-Nesbitt-Nakayama Theorems
 \cite{Skowronski-Yamagata}*{Chapter IV}.
\subsection{The Nakayama automorphism of a ribbon graph order}
In \cite{Butler-Ringel}
Butler and Ringel defined a 
\emph{polarization} for any string algebra. 
Their notion can be adapted to the setup of ribbon graph orders as follows\footnote{The author would like to thank William Crawley-Boevey
for pointing out the relationship between the polarization of a string algebra and Definition \ref{dfn:polarization}.}:
\begin{dfn}\label{dfn:polarization}
A map $\varepsilon\!: Q_1 \to \{+,-\}$ will be called \emph{a polarization of the ribbon graph order $\A$} if for any two different arrows $a, b \in Q_1$ 
which start at the same vertex 
it holds that $\varepsilon(a) \neq \varepsilon(b)$.\\
In the following we will denote 
$\varepsilon(a)$ by $\varepsilon_a$ for any arrow $a \in Q_1$.
\end{dfn}
We will use polarizations to define a sign map 
on the ribbon graph order:
\begin{dfn}\label{dfn:involution}
For any polarization $\varepsilon$ of $\A$  
we define its \emph{involution} $\nuu: \A \to \A:$
\begin{enumerate}
\item 
For any arrow $a \in Q_1$ and any element $x \in \A$ we set 
$
\varepsilon_a \, x = \begin{cases}
\phantom{-}x & \text{if }\varepsilon_{a} = +,\\
-x & \text{if }\varepsilon_{a} = -.
\end{cases}
$
\item 
The \emph{involution} $\nuu: \A \to \A$ associated to $\varepsilon$ is the unique $\Rx$-algebra
 morphism 
such that 
 $\nuu(e_i)= e_i$ for any vertex $i \in Q_0$
and
$$
\nuu(a) = \varepsilon_{\sigma(a)} \, \varepsilon_{a} \, a
\quad \text{ for any arrow $a \in Q_1$}.
$$
Let us recall that $\sigma(a)$ denotes the unique arrow in $Q_1$ such that $\sigma(a) \, a \notin I$.
\end{enumerate}
In the setup above, let
 $\A_{\nuu}$ denote the $\A$-bimodule with the regular left $\A$-module structure and the right $\A$-module structure given by 
$ a \star b = a \, \nuu(b)$ for any $a,b \in \A$.
\end{dfn}

\begin{ex}
The ribbon graph order $\A = \kk \langle\! \langle a,b \rangle \! \rangle/(a^2,b^2)$
has a polarization 
$\varepsilon\!: Q_1 \to \{+,-\}$ with $a \mapsto +, b \mapsto -$.
For any path $p$ in the quiver of $\A$
it holds that $\nuu(p)
= (-1)^{\ell(p)} \, p$, where $\ell(p)$ denotes the length
of the path $p$.
\end{ex}
Let us already point out that 
 the case of characteristic two is special:
\begin{rmk} \label{rmk:polarization}
For any polarization $\varepsilon$ of the order $\A$ the following holds: 
\begin{itemize}
\item If the base field $\kk$ of the order $\A$ has characteristic two, then $\varepsilon_a \, x = x$ for all $a \in Q_1$ and all $x \in \A$, although the map $\varepsilon$ is not constant. In this case, $\nuu = \id_\A$.
\item Otherwise we may view any polarization as a map $\varepsilon: Q_1 \to \{ 1_\kk, -1_\kk\}$.
\end{itemize}
\end{rmk}

Adapting terminology from the theory of finite-dimensional self-injective algebras \cite{Skowronski-Yamagata}*{Chapter IV}, the involution $\nuu$ of a polarization might be called a \emph{Nakayama automorphism} of the 
ribbon graph 
order $\A$ because 
of the following result:
\begin{thm} \label{thm:A}
Let $\nuu$ be the involution of any polarization of the ribbon graph order $\A$. 
Then there is an isomorphism
 $\A^\vee \cong \A_\nuu$ of $\A$-bimodules.
\end{thm}
The proof of Theorem $\ref{thm:A}$ will occupy the next section.
Let us give a few remarks:

In the setup above,
there is also an isomorphism $\A_\nuu \cong 
\mathstrut_\nuu \A$ of $\A$-bimodules,
where $\mathstrut_\nuu \A$ denotes the $\A$-bimodule with the  left module structure twisted by $\varepsilon^*$.

In brief terms, any polarization yields a Nakayama automorphism of the ribbon graph order.
This automorphism may be trivial for any base field:
\begin{ex}
The nodal singularity $\A = \kk \llbracket x,y \rrbracket/(xy)$ has a polarization 
$\varepsilon\!: Q_1 \to \{+,-\}$ with $x \mapsto +, y \mapsto -$. Its involution $\nuu$ is the identity map on $\A$.\\
Theorem~\ref{thm:A} implies the well-known fact that the curve singularity $\A$ is symmetric.
\end{ex}

As a consequence of Theorem \ref{thm:A}
we obtain the first main result of this paper:
\begin{cor}\label{cor:A}
In the notations 
above
the following statements hold:
\begin{enumerate}
\item \label{cor:A1} 
The
derived Nakayama 
functor $\nu$
of the ribbon graph order $\A$
is isomorphic to
the ordinary tensor product functor
$$
\begin{td} 
(
\mathstrut_\text{\textemdash})_{\nuu} = 
\A_\nuu \, {{\otimes}}_{\A}
\mathstrut_{\text{\textemdash}} :
 \D^-(\A) \ar{r}{\sim} \&  \D^-(\A). \end{td}$$
So there are isomorphisms of functors $\nu^2 \cong \Id$ and $\mathbb{S}^2 \cong [2]$
, where $\mathbb{S} = \nu \circ [1]$.
\item \label{cor:A3}  For any projective complex $\CP$ from $\D^-(\A)$
the complex $\nu(\CP)$ can be computed by application of the involution $\nuu$  
 to each entry in every
differential of the complex $\CP$.
\end{enumerate}
\end{cor}
\begin{proof}
\begin{enumerate}
\item
By Theorem \ref{thm:A} 
the
Nakayama
functor
is isomorphic to the tensor product
$\begin{td}\A_\nuu \otimes_{\A} \mathstrut_{\text{\textemdash}}: \md \A \ar{r}{\sim} \& \md \A \end{td}$.
Since there is an isomorphism of $\A$-bimodules $\A_\nuu \otimes_\A \A_\nuu \cong \A_{(\nuu)^2} = \A$,
the Nakayama functor is 
an equivalence of order two, and thus it is exact. 
This shows the statements in \ref{cor:A1}.
\item
For any path $p$ from $i$ to $j$ in the quiver $(Q,I)$
 there is a commutative diagram:
$$
\begin{td} \A_\nuu \otimes_{\A} P_j \ar{d}{\id \otimes \cdot p} \ar{r}{\sim} \& \A \nuu(e_j) = P_j \ar[xshift=-5pt,dashed]{d}{\cdot \nuu(p)} \& a \otimes e_j \ar[mapsto]{d}  \&  \ar[mapsto]{l} a \nuu(e_j) \ar[dashed, mapsto]{d} \\
\A_\nuu \otimes_{\A} P_i \ar{r}{\sim} \& \A \nuu(e_i) = P_i \& a \otimes p \ar[mapsto]{r} \& a \nuu(p)
\end{td}
$$
This implies the last statement. \qedhere
\end{enumerate}
\end{proof}
The derived Nakayama functor $\nu$ restricts to the subcategory $\Tau = \Perffd(\A)$.
By Corollary \ref{cor:A} \ref{cor:A1}, the category $\Tau$ 
is \emph{fractionally $\frac{2}{2}$-Calabi-Yau}
and its
 Auslander-Reiten quiver 
is given by homogeneous tubes of rank one or two. 

Since the restricted Serre functor $\mathbb{S}$ is isomorphic to 
the composition
$(\mathstrut_{\text{\textemdash}})_\nuu \circ [1]$  
the category $\Tau$
is \emph{twisted $1$-Calabi-Yau} in the sense of Herschend and Iyama  \cite{Herschend-Iyama}.

By Corollary \ref{cor:A} \ref{cor:A3},
any ribbon graph order $\A$ is \emph{weakly symmetric}, that is, there is an isomorphism $\omegac \otimes_{\A} P \cong P$ for any indecomposable projective $\A$-module $P_i$.

\subsection{Proof of Theorem \ref{thm:A} on the canonical bimodule}
Additionally to the setup at the beginning of Section \ref{sec:canonical}, we choose 
\begin{itemize}[label=$\blackdiamond$]
\item a polarization $\varepsilon\!: Q_1 \to \{+,-\}$ of the ribbon graph order $\A$
as introduced in Definition \ref{dfn:polarization} and denote by $\nuu$ its involution.
\end{itemize}

\subsubsection{Bases of the ribbon graph order and its canonical bimodule}

To choose a good basis of 
the module $\mathstrut_{\Rx} \A$ 
we  need to fix some notation.
The first part concerns paths of certain length:
\begin{notation}
 For any arrow $a$ of the complete gentle quiver $(Q,I)$ we will denote:
\begin{itemize}[label=$\blackdiamond$]
\item As before,
let
$\sigma(a) \in Q_1$ denote
the unique ``successor'' of $a$ such that $\sigma(a)\, a\notin I$.
\item Again, $n(a)$
is the length of the repetition-free cycle $c_a$ starting with 
the arrow  $a$.
\item From now on, for any number $m \in \N^+$
let $a_m$ denote the unique path of length $m$ which starts with the arrow $a$. In other terms, $a_m = \sigma^{m-1}(a) \ldots \sigma(a) \, a$. 
\end{itemize}
Note that
$ a_{n(a)} = c_a$. We need some notation for the following subset of arrows:
\begin{itemize}[label=$\blackdiamond$]
\item We denote by ${Q}^*_1 = \{a \in Q_1 \ | \ c_a \neq a \}$ 
the set of arrows which are not cycles.
\end{itemize}

\end{notation}

\begin{rmk}
For any arrow $a \in Q_1$ the following conditions are equivalent:
\begin{itemize} 
\item[] $(\mathsf{1})$ $ c_a = a$,
\qquad $(\mathsf{2}) \  \sigma(a) = a$,
\qquad $(\mathsf{3}) \ n(a) = 1 $,
\quad $(\mathsf{4}) \  s(a) = t(a)\text{ and }a^2 \notin I$,
\item[] $(\mathsf{5})$
The node $v_a$ of the graph $\UG$  of the complete gentle quiver $(Q,I)$ is a \emph{leaf},\\
 $\phantom{(\mathsf{5})}$ that is, it is incident to only one edge and this edge is ordinary.
\end{itemize}
\end{rmk}
The second part of notation involves the polarization $\varepsilon^\star$ of $\A$ fixed above:
\begin{notation}\label{not:polarized}
For any vertex $i\in Q_0$ we define two cyclic paths $x_i$ and $y_i:$
\begin{itemize}[label=$\blackdiamond$]
\item Set $x_i = c_a$ where $a \in Q_1$ is the unique arrow such that $s(a) = i$ and $\varepsilon_a = 1.$
\item Similarly, set $y_i = c_b$ where $b$ is the unique arrow with $s(b) = i$ and $\varepsilon_b = -1$.
\end{itemize}
Recall that
we use the
 $\Rx$-algebra structure $\Rx \to \A$, $t \mapsto z := \sum_{a\in Q_1} c_a$.
Note that $z e_i = e_i z = x_i + y_i$ for any vertex $i \in Q_0$.

\end{notation}
With the notation above, we may describe a basis of the $\Rx$-modules $\mathstrut_\Rx \A$ and $\mathstrut_{\Rx} \A^\vee:$ 
\begin{lem}\label{lem:basis}
Let $(Q,I)$ be some complete gentle quiver, $\A$ its ribbon graph order and $\varepsilon$ some polarization.
Then the free module  $\mathstrut_{\Rx} \A$ has an $\Rx$-linear basis given by
$$
\mathcal{B}
 = \{ e_i, \ x_i \ | \ i \in Q_0\} 
\cup \{a_m \ | \ a\in Q^*_1, \ 1 \leq m < n(a) \}.$$
In particular, the module $\mathstrut_{\Rx} \omegac$ has 
an $\Rx$-linear basis $\mathcal{B}^\vee := \{p^\vee \ | \ p \in \mathcal{B}\}$
where $p^\vee\!: \A \to \Rx$ is the unique $\Rx$-linear morphism given by
$p^\vee(q) = \delta_{pq}$ for any $q \in \mathcal{B}$.
\end{lem}
\begin{proof}
Let $p$ be some path of $(Q,I)$ starting with some arrow $a \in Q_1$.
\begin{itemize}
\item
If $p$ is not cyclic, then $p = a_{r \,n(a) +m} = z^r a_m$ for some $r \in \N^+$ and $1 \leq m < n(a)$. 
\item If $p$ is cyclic, then $p = c_a^r = z^{r-1} x_i$ or $z^{r-1}y_i$ for some $r \in \N^+$, where $i = s(a)$.
\end{itemize}
In particular, $\mathcal{B}$ forms a set of generators for $\mathstrut_{\Rx} \A$.
Since $|Q_1|  = 2 \, |Q_0|$, Proposition \ref{prp:RGO} \ref{prp:RGO2} implies that $|\mathcal{B}| = \rk_\Rx \A$.
So the set $\mathcal{B}$ is an $\Rx$-linear basis of $\mathstrut_{\Rx} \A$.
\end{proof}
\begin{ex}
\begin{enumerate}
\item Let $\A = \kk \llbracket x,y \rrbracket/(xy)$
 with polarization $\varepsilon_x = +$ and $\varepsilon_y = -$.
Then $z =  x+y$. Lemma \ref{lem:basis} yields the basis $\mathcal{B} = \{1, x\}$ for the module $\mathstrut_\Rx \A$.
\item Let
 $\A = \kk \langle\!\langle a,b\rangle \!\rangle/(a^2,b^2)$ with polarization $\varepsilon_a = +$ and $\varepsilon_b = -$.
Then $z = ba + ab = x_1 + y_1$.
By 
Lemma \ref{lem:basis} the module $\mathstrut_\Rx \A$ has a basis $\mathcal{B} = \{1, ba, a, b\}$.
\end{enumerate}
\end{ex}

\subsubsection{The Frobenius form of a ribbon graph order}
Using the dual basis fixed above we 
 introduce a special $\Rx$-linear form of a ribbon graph order:
\begin{dfn}\label{dfn:frobenius}
For a ribbon graph order $\A$ and a polarization $\varepsilon$ of $\A$ we set
\begin{align*}
\omega := 
\sum_{i \in Q_0} x_i^\vee=\sum_{a \in Q_1: \, \varepsilon_a = +} c_a^\vee  && \text{and}&&
\phi := \omega^\vee: \begin{td} \A \ar{r} \& \Rx. \end{td}
\end{align*}
The $\Rx$-linear form $\phi$ will be called
the \emph{Frobenius form} of $\A$.
\end{dfn}
This definition is motivated by an analogue for Brauer graph algebras which will be discussed in Remark \ref{rmk:frobenius-form}.

Next, we compute the Frobenius form on compositions of basis elements.
\begin{notation}
For any $a\in Q_1^*$ and $1 \leq m < n(a)$ we 
define a ``derivative cycle''
\begin{itemize}[label=$\blackdiamond$]
\item
$\partial^m c_a = \sigma^m(a)_{n(a)-m} = \sigma^{n(a)-1}(a) \ldots \sigma^{m+1}(a) \, \sigma^m(a)$. Thus, $(\partial^m c_a) \, a_m = c_a$.
\end{itemize}
\end{notation}
\begin{lem}\label{lem:relations}
For any $p,q \in \mathcal{B}$ it holds that $\phi(qp) \neq 0$
if and only if 
$$
(q,p) \in \{ (x_i,x_i), (x_i,e_i), (e_i,x_i) \ | \ i \in Q_0 \}
\cup \{ (\partial^m c_a, a_m) \ | \ a \in Q_1^*, \ 1 \leq m < n(a)\}.
$$ 
\end{lem}
\begin{proof}
\begin{itemize}[leftmargin=25pt]
\item[$\Leftarrow:$] 
Let $i \in Q_0$.
We note that
$\phi(x_i^2) = \phi(z x_i) =  z \neq 0$, $\phi(x_i) = 1$ and $\phi(y_i) = \phi(z e_i - x_i) = -1$, so $\phi(c_a) = \varepsilon_a \neq 0$ for any arrow $a \in Q_1$.
\item[$\Rightarrow:$]  
Let $p,q \in \mathcal{B}$ such that $\phi(qp)\neq 0$. 
Then 
$qp = a_\ell$ for some arrow $a \in Q_1$ and some number  $1 \leq \ell \leq 2 \, n(a)$. Since $qp$ must be a cyclic path
it holds that $qp = c_a^2$ or $c_a$.
So $qp \in \{x_i^2, x_i, y_i\}$ with $i = s(a)$.
This implies the possibilities above.
\qedhere
\end{itemize}
\end{proof}
The last lemma implies that the Frobenius form $\phi$ is an \emph{$\nuu$-symmetrizing} form:
\begin{lem} \label{lem:nu-symmetry}
For any $p,q \in \A$ it holds that $\phi(qp) = \phi(\nuu(p) q)$.
\end{lem}
\begin{proof}
For any arrow $a\in Q_1$ it holds that $\nuu(c_a) = c_a$.
So  $\nuu(z) = z$.
Since the involution $\nuu$ and the form $\phi$ are both $\Rx$-linear,
it is sufficient to show the claim  for any two  elements $p,q \in \mathcal{B}$.
\begin{itemize}
\item Assume that $\phi(qp) \neq 0$. By Lemma \ref{lem:relations} there are only two possibilities: 
\begin{itemize}[label=$\bullet$]
\item If $(q,p)= (x_i, x_i), (x_i,e_i)$ or $(e_i,x_i)$ for some $i \in Q_0$, then   $\nuu(p) q = pq = qp$.
\item If $(q,p)=  (\partial^m c_a,a_m)$ for some arrow $a \in Q_1^*$ and some number $1 \leq m <n(a)$, then  $\nuu(p) q= \varepsilon_{\sigma^m(a)} \, \varepsilon_a \, c_{\sigma^m(a)}$ and $\phi(\nuu(p)q) = 
\varepsilon_{\sigma^m(a)}^2 \, \varepsilon_a = \varepsilon_a = \phi(c_a) = \phi(qp)$.
\end{itemize}
In particular, we have shown that $\phi(\nuu(p) q) = 0$ implies that
$\phi(qp) = 0$.
\end{itemize}

\begin{itemize}
\item Let $\phi(qp)=0$. 
Then $0 
= \phi((\nuu)^2(qp)) = \phi((\nuu)^2(p) \nuu(q)) = \phi(\nuu(p) q)$. \qedhere
\end{itemize}
\end{proof}

The $\A$-bimodule structure of the canonical bimodule $\omegac = \Hom_{\Rx}(\A,\Rx)$ is given as follows. For any form $\gamma \in \omegac$ and any two elements $a,b \in \A$ the forms
$a \cdot \gamma$
 and $\gamma \cdot b$ 
are
 defined via
$(a \cdot \gamma )(x) = \gamma(xa)$
 and $(\gamma\cdot b)(x) = \gamma(bx)$
for any element $x \in \A$.

At last, we prove a refined version of Theorem \ref{thm:A}:
\begin{manualtheorem}{\ref{thm:A}$^\#$}\label{thm:A'}
Let $\varepsilon$ be some polarization of the ribbon graph order $\A$.
 Let $\nuu$ denote its involution and $\phi$
 its Frobenius form.
 \\
Then the map $\vartheta\!: \A_\nuu \to \A^{\vee},  \ p \mapsto p \cdot \phi$ is an isomorphism of $\A$-bimodules.
\end{manualtheorem}
\begin{proof}
Let us recall that
the Frobenius form is given by $\phi
 = \sum_{i \in Q_0} x_i^\vee\!: \A \to \Rx$,
and the
 right $\A$-module structure of $\A_\nuu$ by
$\A_\nuu \times \A \to \A_\nuu$, 
$(p,b) \mapsto 
p \star b := p \,\nuu(b)$.
\begin{enumerate}
\item 
The map $\vartheta$ is left $\A$-linear by definition.
Let $p \in \A$. 
Lemma \ref{lem:nu-symmetry} yields that
$p \cdot \phi
= \phi \cdot \nuu(p)$
and thus $\nuu(p) \cdot \phi = \phi \cdot p$.
For any $x \in \A$ it follows that
$$\vartheta(x \star p) = \vartheta(x \, \nuu(p)) = (x\,\nuu(p)) \cdot \phi
= x \cdot (\phi \cdot p) = (x\cdot \vartheta(1))\cdot p
= \vartheta(x) \cdot p.
$$
This shows that the map $\vartheta$ is a morphism of $\A$-bimodules.
\end{enumerate}
To show that $\vartheta$ is bijective, we first compute $\vartheta(p)$ for all elements $p \in \mathcal{B}$:
\begin{enumerate} \setcounter{enumi}{1}
\item 
We note that
any 
$\Rx$-linear morphism 
 $\psi\!:
\A \to \Rx$ can be written as 
$\psi = \sum_{q \in \mathcal{B}} 
\psi(q) \, q^\vee$.
For any vertex
 $i \in Q_0$ Lemma \ref{lem:relations} implies the following:
\begin{align*}
\vartheta(e_i)
= e_i \cdot \phi 
= 
\sum_{q \in \mathcal{B}} \phi(q e_i) \, q^\vee
= 
x_i^\vee&&
\vartheta(x_i)
= x_i \cdot \phi = 
\sum_{q \in \mathcal{B}} 
\phi(q x_i) \, q^\vee
= 
e_i^\vee + z x_i^\vee.
\end{align*}
For any $a \in Q_1^*$ and any number $1 < m < n(a)$ 
 Lemma \ref{lem:relations} also yields that
\begin{align*}\vartheta(a_m) = a_m \cdot \phi = 
\sum_{q \in \mathcal{B}} 
\phi(q a_m) \cdot q^\vee
= \phi(c_a) (\partial^m c_a)^\vee
=
\varepsilon_a (\partial^m c_a)^\vee.\end{align*}
\item 
Let $\psi: \A^{\vee} \to \A_\nuu$ be the unique $\Rx$-linear map given by
\begin{align*}
x_i^\vee &\mapsto e_i 
\quad \text{and} \quad e_i^\vee \mapsto -y_i = x_i - z e_i
\quad
\text{for any vertex $i\in Q_0$}, \\
a_m^\vee &\mapsto \varepsilon_{\sigma^m(a)} \, 
\partial^m c_a
\quad
\text{for any arrow $a \in Q_1^*$ and any number $1 \leq m < n(a)$}. 
\end{align*} 
It is straightforward to check that $\psi$ is the inverse of the map $\vartheta$. 
\qedhere
\end{enumerate}
\end{proof}

\section{Symmetric ribbon graph orders and Brauer graph algebras}\label{sec:symmetric}

So far, we know that the canonical bimodule of the ribbon graph order $\A$
is given by a bimodule $\A_\nuu$, where the left module structure is the regular one and the right module structure is twisted by some involution $\nuu$. 

The goal of this section is to describe conditions when this twist ``does not matter''. 
In Subsection \ref{subsec:push},
we observe a few statements on
central quotients of some algebras which seem to be interesting on their own.
In the remaining subsections we  collect the implications to  characterize symmetric ribbon graph orders:
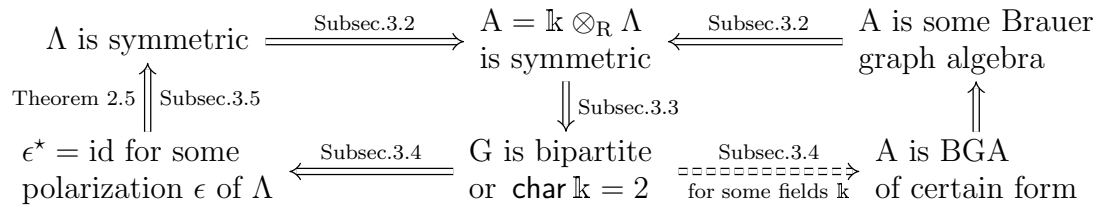
\begin{align*}
\begin{tikzcd}[ampersand replacement=\&, cells={outer sep=1pt, inner sep=0pt}, column sep=2.25cm, labels={outer sep=2pt},
]
{\begin{array}{c}
\A \text{ is symmetric}
\end{array}} 
\ar[Rightarrow]{r}{
\text{Subsec.\ref{subsec:twistedBGA}}}
\&
\begin{array}{l}
 \AI = \kk \otimes_{\Rx}\A\\
 \text{is symmetric} 
\end{array}
\ar[Rightarrow, xshift=-0.0cm]{d}[xshift=0.0cm]{\text{Subsec.\ref{subsec:circuits}}} 
\&
\ar[Rightarrow]{l}[swap]{\text{
Subsec.\ref{subsec:twistedBGA}}
} 
\begin{array}{l}
\AI \text{ is some Brauer}
\\
\text{graph algebra}
\end{array}
\\
\begin{array}{l}
\text{$\epsilon^\star = \id$ for some}\\
\text{polarization } \epsilon \text{ of }\A
\end{array}
\ar[Rightarrow]{u}{\text{Theorem  \ref{thm:A}} }[swap]{\text{Subsec.}\ref{subsec:final}}
\&
\begin{array}{l}
 \UG\text{ is bipartite  }\\
\text{or }\chr \kk =2
\end{array}
\ar[Rightarrow]{l}[swap]{\text{Subsec.\ref{subsec:bipartite}}}
\ar[dashed,Rightarrow,
end anchor=west, start anchor= east
]{r}{\text{Subsec.\ref{subsec:bipartite}}}[swap]{
\text{for some fields } \kk
}
\&
\begin{array}{l}
\AI \text{ is BGA }
\\
\text{of certain form}
\ar[Rightarrow]{u}
\end{array}
\end{tikzcd}
\end{align*}
The precise version of the diagram above is stated in Theorem \ref{thm:symmetry}.
In Subsection \ref{subsec:final}
we will also clarify the relationship between certain natural quotients of ribbon graph orders and
Brauer graph algebras.

\subsection{Pushing down the canonical bimodule}
\label{subsec:push}
In this brief subsection, we consider a more general setup than usual:
\begin{itemize}[label=$\blackdiamond$]
\item
Let $\Sx$
be a commutative ring,
 $\mx$ any maximal ideal of $\Sx$
and $\A$ be an $\Sx$-algebra  
 such that $\mathstrut_\Sx \A$ is a free $\Sx$-module of finite rank.
\end{itemize}
For example, $\Sx$ might be a regular local ring and $\A$ an $\Sx$-order.
\begin{itemize}[label=$\blackdiamond$]
\item 
We set  $\A^{\vee} = \Hom_{\Sx}(\A,\Sx)$,
$\ \kk = \Sx/\mx$, $\ \AI = \A/\mx \A$ 
 and $\ \omegad = \Hom_{\kk}(\AI,\kk)$.
\end{itemize}
Let us note that
there is an isomorphism of $\kk$-algebras $\AI \cong \kk \otimes_{\Rx} \A$ and that $\mx \A = \A \mx$ is a two-sided ideal of $\A$.

The canonical bimodule $\omegad$ is a quotient of the canonical bimodule $\A^{\vee}$:
\begin{lem}\label{lem:bimodule-quotient}
 In the setup above,
 there is an isomorphism 
of $(\AI, \A)$-bimodules
$$
\begin{td}
\phi\!:
\AI \otimes_{\A} 
\A^{\vee} \ar{r} \& 
\omegad
, \& \ol{a} \otimes_{\Sx} f \ar[mapsto]{r} \& \ol{a f},
\end{td}
$$
where $\ol{af}$ maps an element
$x + \mx \A \in \AI$ to the scalar $f(x\,a) + \mx \in \kk$ for any $x \in \A$.
\end{lem}
\begin{proof} 
It is straightforward to verify that the map $\phi$ is left $\AI$-linear and right $\A$-linear.
Since the functor $\AI \otimes_{\A} \mathstrut_{\text{\textemdash}}:
\begin{td} \md \A \ar{r} \& \md \AI \end{td}$ is full, 
the composition
 $$ \begin{td} 
\A^{\vee}
 \ar[->>]{r} \& \AI \otimes_{\A} 
\A^{\vee}
 \ar{r}{\phi} \& 
\omegad,
\& f \ar[mapsto]{r}\& \ol{1} \otimes f \ar[mapsto]{r} \& \ol{f}
\end{td}
$$
is surjective. Thus, the map $\phi$ is also surjective.
The $\kk$-linear map $\phi$ is also bijective, since
 $\dim_{\kk} (\AI \otimes_{\A} \A^{\vee}) =
\dim_{\kk} (\kk \otimes_{\Sx} \A^{\vee})
= 
\rk_{\Sx} \A^{\vee} = \rk_{\Sx} \A = \dim_{\kk} \AI = \dim_{\kk} \omegad$. 
\end{proof}
Symmetricity \emph{ascends} along the  projection $\A \twoheadrightarrow \AI$ in the sense of \ref{prp:bimodule-quotient2} below:
\begin{prp} \label{prp:bimodule-quotient}
In the setup above, the following statements hold:
\begin{enumerate}
\item \label{prp:bimodule-quotient1}
There is an isomorphism of $\AI$-bimodules
$
\AI \otimes_{\A} \A^{\vee} \otimes_{\A} \AI \cong
\omegad.
$
\item \label{prp:bimodule-quotient2}
If the $\Sx$-algebra $\A$ is symmetric, then the $\kk$-algebra $\AI$ is symmetric.
\end{enumerate}
\end{prp}
\begin{proof}
By Lemma \ref{lem:bimodule-quotient} there is an isomorphism
$\AI \otimes_{\A} \A^{\vee} \otimes_{\A} \AI \cong \omegad \otimes_{\A} \AI \cong   \omegad$ of $\AI$-bimodules which shows the first claim. The first claim implies the second.
\end{proof}
In the next subsection, we will apply these results to ribbon graph orders.
\subsection{Twisted Brauer graph algebras} \label{subsec:twistedBGA}

Throughout this subsection, we return to our original setup and extend it as follows:

\begin{notation}\label{not:BGA}
Let $\A$ be a ribbon graph order.
\begin{itemize}[label=$\blackdiamond$]
\item  Let $(Q,I)$ denote a complete gentle quiver of $\A$ and $\UG_0$ the set of its $\sigma$-orbits. For any arrow $a \in Q_1$ its $\sigma$-orbit   is denoted by $v_a$ as before.
\item Let $m\!: \UG_0 \to \N^+$ be some map.
 For any arrow $a \in Q_1$ we set $m_a = m(v_a)$.
\item
We view the ring $\A$ as an $\Rx$-order via the map
 $\Rx \to \A$, $t \mapsto z = \sum_{a\in Q_1} c_a^{m_a}$ from Proposition \ref{prp:RGO}. We define a quotient ring $\AI = \A/ z \A$.
%
%
%
%
%
%
%
%
%
%
%
%
%
%
%
%
%
\item
Let
$\varepsilon$ be any polarization of $\A$.
Using the map $m$ above we define 
another element $w = \sum_{a \in Q_1} \varepsilon_a^{\mathstrut}\, c_a^{m_a}$
and another quotient ring $\BI = \A/\A w \A$.
\end{itemize}
We will write $m \equiv 1$ if $m(v) = 1$ for any $v \in \UG_0$.
\end{notation}
\begin{ex}
Let $\A = \kk \langle \! \langle a,b \rangle \! \rangle/(a^2,b^2)$ and $m \equiv 1$.
Then $\AI = \A/(ba+ab)$ 
is an anti-commutative algebra,
while
$\BI \cong \kk \llbracket a,b \rrbracket / (a^2,b^2)$
is commutative.
\end{ex}
\begin{lem}\label{lem:BGA-relations}
The quotient rings $\AI$ 
and $\BI$ from Notation
\ref{not:BGA}
are Morita equivalent to the 
finite-dimensional
 path algebras $\kk Q/(I+J_+)$
respectively $\kk Q/ (I+J_-)$
where 
$$
J_\pm = 
 \left\langle  c_{a\mathstrut}^{m_a} \pm  c_b^{m_b} \ \big| \ a,b \in Q_1\!: s(a) = s(b)
\right\rangle_{\kk Q}
$$
In particular, the algebra $\BI$ does not depend on the choice of polarization $\varepsilon$
and the socle quotients
$\AI/\soc \AI$ and $\BI/\soc \BI$ are Morita equivalent.
\end{lem}
\begin{proof}
For any $a \in Q_1$ it holds that $
a_{n(a) \, m_a + 1} = a \, c_a^{m_a} = \mp a\, c_b^{m_b} \in I$. So each algebra $\kk Q/(I+J_\pm)$ has finite dimension.
The other claims follow directly.
\end{proof}
\begin{rmk}
In the notation above, both ideals $I + J_\pm$ are not admissible
if there is an arrow $a \in Q_1$ such that $c_a = a$ and $m_a = 1$.
\end{rmk}

We will be interested mainly in the quotients 
 with  anti-commutativity relations:
\begin{dfn}\label{dfn:twisted}
A finite-dimensional $\kk$-algebra $\AI$ will be called a \emph{twisted Brauer graph algebra}
if 
there 
is a complete gentle quiver $(Q,I)$ 
and a map $m: \UG_0 \to \N^+$ on its $\sigma$-orbits
such that $\AI$ is Morita equivalent to 
the path algebra $\kk Q /(I + J_+).$
\end{dfn}

The terminology above will be motivated at the end of the present subsection.

Replacing the ideal $J_+$ by $J_-$
in Definition \ref{dfn:twisted}
yields the  definition of a usual \emph{Brauer graph algebra}.
We will describe when a Brauer graph algebra is isomorphic to its twisted version in Subsection \ref{subsec:bipartite}. 
The last statement of Lemma \ref{lem:BGA-relations} shows that any Brauer graph algebra is a \emph{socle deformation} of its twisted version.

The next property of Brauer graph algebras is well-known:
\begin{thm} \label{thm:BGA}
Any Brauer graph algebra $\BI$ is symmetric.
\end{thm}
%
%
%
The following
proof 
is due to Schroll
\cite{Schroll18}*{Proof of Theorem 2.6}:
\begin{proof}
We use Notation \ref{not:BGA}. It can be verified that the element $\ol{\omega} = \sum_{a \in Q_1: \varepsilon_a = +}
{\ol{c_a}}^{m_a}$ of $\BI$ yields
a symmetrizing $\kk$-linear form $\phi={\ol{\omega}}^* \in \BI^*$
such that $\phi(I) \neq 0$ for any non-zero left ideal $I$ of $\BI$.
Then the form $\phi$ induces
an isomorphism of $\BI$-bimodules
 $\BI \to \BI^*$ via $x \mapsto x \cdot \phi$ 
 \cite{Skowronski-Yamagata}*{Chapter IV, Theorem 2.2}.
\end{proof}
\begin{rmk}
\label{rmk:frobenius-form}
Let $m \equiv 1$. Then $\A$ is an $\Rx$-order via the map $\Rx \to \A$, $t \mapsto \sum_{a \in Q_1} c_a$.\\
On the one hand, the Frobenius form $\phi = \omega^\vee$ of the ribbon graph order $\A$ from Definition \ref{dfn:frobenius} 
is essentially the same as
the symmetrizing Frobenius form ${\ol{\omega}}^*$ of the Brauer graph algebra $\BI=\A/\A w\A$ from the last proof. 
\\
On the other hand, we will see that the ribbon graph order $\A$ may be not symmetric. 
In this case, 
 the map
$\phi: \Rx \to \, \A$, $t\mapsto w$
does not factor through the center of $\A$.
\end{rmk}
In contrast to Brauer graph algebras, their twisted versions 
fit into the framework
 of the previous subsection:
\begin{lem}\label{lem:algebra-structure}
Any twisted Brauer graph 
$\kk$-algebra $\AI$
is isomorphic to the quotient algebra $\kk \otimes_{\Rx} \A$
for some ribbon graph order $\A$.
\end{lem}
\begin{proof}
The ring
$\AI$ is isomorphic to the quotient of some ribbon graph order $\A$
by Definition \ref{dfn:twisted} with respect to some multiplicity map $m$.
By Proposition \ref{prp:RGO},
the  map $m$
yields an $\Rx$-algebra morphism $
\Rx \to \A$, $t \mapsto z=\sum_{a \in Q_1} c_a^{m_a}$.
This implies that $\AI \cong \A/z \A \cong \kk \otimes_{\Rx} \A$.
\end{proof}

Next, we obtain an analogue of Theorem \ref{thm:A}
for twisted Brauer graph algebras and the first implication in our characterization of symmetric ribbon graph orders:
\begin{thm} \label{thm:twisted-BGA}
Let $\A$ be a ribbon graph order and $m: \UG_0 \to \N^+$ be a map on its $\sigma$-orbits and $\varepsilon$ be any polarization of $\A$.
Let $\AI$ denote the corresponding twisted Brauer graph algebra from Notation
\ref{not:BGA}.
Then the following statements hold:
\begin{enumerate}
\item \label{thm:twisted-BGA2}
The map $
\begin{td} \olnuu\!: \AI \ar{r} \& \AI \end{td}, \begin{td}
{\ol{x}} \ar[mapsto]{r} \& {\ol{\nuu(x)}} 
\end{td}$
is a well-defined $\kk$-algebra involution.\\
%
%
%
%
%
%
%
%
%
Moreover,
there is an isomorphism
$\omegad \cong \AI_{\scriptsize \olnuu}$ of $\AI$-bimodules.
\item \label{thm:twisted-BGA3}
If the ribbon graph order $\A$ is symmetric
or if the $\kk$-algebra $\AI$ is isomorphic to some Brauer graph algebra, then the $\kk$-algebra
$\AI$ is symmetric.
\end{enumerate}

\end{thm}

\begin{proof}
\begin{enumerate}
\item 
 Note that $\nuu(c_a) = c_a$ for
any arrow $a \in Q_1$. This implies that $\nuu$ preserves the ideal $z \A =  \A z$.
Since $\nuu$ is an involution of $\Rx$-algebras, the first statement follows.\\
%
%
%
%
%
 Proposition \ref{prp:bimodule-quotient} and Theorem \ref{thm:A} yield two isomorphisms of $\AI$-bimodules:
$$
\omegad 
\cong
\AI \otimes_{\A} \omegac \otimes_{\A} \AI
\cong
\begin{td} {\AI \otimes_{\A} \A_{\nuu} \otimes_{\A} \AI} \ar{r}{\eta} \& \AI_{\scriptsize \olnuu},\end{td} \quad
 \eta({\ol{a} \otimes b \otimes \ol{c}}) =
{\ol{a \, b \,\nuu(c)}} 
$$
It is straightforward to check that $\eta$ 
is also an isomorphism
 of $\AI$-bimodules.
\item 
\begin{itemize}
\item Assume that the order $\A$ is symmetric.
By Lemma \ref{lem:algebra-structure}
we may apply Proposition \ref{prp:bimodule-quotient} \ref{prp:bimodule-quotient2} to the setup above and conclude that $\AI$ is symmetric.
\item If the $\kk$-algebra $\AI$ is isomorphic to
some Brauer graph algebra, 
then $\AI$  is symmetric by Theorem \ref{thm:BGA}.
 \qedhere
\end{itemize}
\qedhere
\end{enumerate}
\end{proof}

The description of the canonical bimodule $\omegad$ of a twisted Brauer graph algebra $\AI$ yields similar consequences as for ribbon graph orders:

The category $\Tau = \Perf(\AI)$ 
has a Serre functor $\mathbb{S} = \AI_{\scriptsize \olnuu} \otimes_{\AI} (\mathstrut_\text{\textemdash})\!:\begin{td} \Tau \ar{r}{\sim} \& \Tau \end{td}$.
Because of the isomorphism $\mathbb{S}^2 \cong \Id$ of functors,
  the category $\Tau$ 
is \emph{fractionally $\frac{0}{2}$-Calabi-Yau}.
Since $\olnuu$ is an involution, the category $\Tau$ is also \emph{twisted $0$-Calabi-Yau}.\\
As any idempotent of $\AI$ is a fixed point of the map $\varepsilon^*$,
the twisted Brauer graph algebra $\AI$ is \emph{weakly symmetric}.

\subsection{Circular graph orders and circular graph algebras}
\label{subsec:circuits}

Next we study the minimal non-symmetric twisted Brauer graph algebras.

By a \emph{circular graph} of length $n \geq 2$ we mean a 
 connected graph with $n$ vertices of valency two.
A circular graph of length one is a loop.

In general, there can be many non-isomorphic ribbon graph orders with isomorphic underlying graphs in the sense of Construction \ref{constr:graph}.
Fortunately, this does not happen for circular graphs:

\begin{lem}\label{lem:circuit-order}
Let $n \in \N^+$ and $\A_n$ be a 
 ribbon graph order 
such that its graph $\UG_n$ is a circular graph of length $n$.
Then $\A_n$ has the following complete gentle quiver:
$$
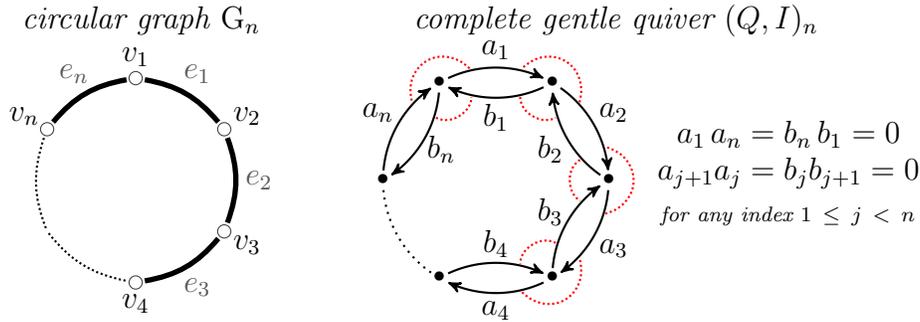
\begin{array}{cc}
\text{circular graph }\UG_n & \text{complete gentle quiver }(Q,I)_n  \\
\begin{tikzpicture}[
baseline=0pt]
\def \n {6}
\def \radius {1.5cm}
\def \margin {4} %
\foreach \s in {1,...,5}
{
\def \phi {-120+360/\n * (\s - 1)};
\def \psi {-120+360/\n * (\s - 0.5)};
\def \xi {-120+360/\n * (\s)};
\def \R {2*\radius*cos(180/\n)};
  \node[rectangle, inner sep=1pt] (c\s) at ({\phi}:\radius) {};
  \node[inner sep=0pt,
draw=black!75, fill=white, circle, outer sep=1pt, minimum size=5pt] (z\s) at ({\psi}:{\radius-\margin}) {$\phantom{.}$};
}
  \node[inner sep=0pt] (z0) at (0:0) {};
\foreach \s in {6}
{
\def \phi {-120+360/\n * (\s - 1)};
\def \psi {-120+360/\n * (\s - 0.5)};
\def \xi {-120+360/\n * (\s)};
\def \R {2*\radius*cos(180/\n)};
  \node[rectangle, inner sep=1pt] (c\s) at ({\phi}:\radius) {};
  \node[inner sep=-1pt] (z\s) at ({\psi}:{\radius-\margin}) {};
;
}
\foreach \s/\t in {1/2,2/3,3/4,4/5}
{
\draw[bend left=20,
 draw=black!20, line width=2pt, inner sep=0pt] (z\t) edge (z\s);
}
{
\draw[densely dotted, thick, bend right=20,draw=black!20] (z6) edge  (z1);
\draw[densely dotted, thick, bend right=20,draw=black!20] (z5)  edge (z6);
}
\node[inner sep=0pt] (v4) at (270:{\radius+\margin}) {$v_4$};
\node[inner sep=0pt] (v3) at (330:{\radius+1.5*\margin}) {$v_3$};
\node[inner sep=0pt] (v2) at (30:{\radius+1.5*\margin}) {$v_2$};
\node[inner sep=0pt] (v1) at (90:{\radius+\margin}) {$v_1$};
\node[inner sep=0pt] (vn) at (150:{\radius+1.5*\margin}) {$v_n$};
\node[inner sep=0pt] (e3) at (300:{\radius+\margin}) {$\textcolor{black!66}{e_3}$};
\node[inner sep=0pt] (e2) at (0:{\radius+\margin}) {$\textcolor{black!66}{e_2}$};
\node[inner sep=0pt] (e1) at (60:{\radius+\margin}) {$\textcolor{black!66}{e_1}$};
\node[inner sep=0pt] (en) at (120:{\radius+\margin}) {$\textcolor{black!66}{e_n}$};
\end{tikzpicture}
&
\qquad
\begin{tikzpicture}[
baseline=0pt]
\def \n {6}
\def \radius {1.5cm}
\def \margin {5} %
\foreach \s in {1,...,5}
{
\def \phi {-120+360/\n * (\s - 1)};
\def \psi {-120+360/\n * (\s - 0.5)};
\def \xi {-120+360/\n * (\s)};
\def \R {2*\radius*cos(180/\n)};
  \node[rectangle, inner sep=1pt] (c\s) at ({\phi}:\radius) {$\bt$};
  \node[inner sep=0pt] (z\s) at ({\psi}:{\radius+\margin}) {};
  \draw[<-,thick,black, >=stealth', domain={{\phi+\margin}}:{{\xi-\margin}}]  plot  
({\radius*cos(\x)}, {\radius*sin(\x)})
;
\draw[blue, inner sep=0pt] ({\radius*cos(\phi-3.5*\margin)}, {\radius*sin(\phi-3.5*\margin)})  node (a\s) {};
\draw[black, inner sep=0pt] ({\radius*cos(\phi+3.5*\margin)}, {\radius*sin(\phi+3.5*\margin)})  node (b\s) {};
  \draw[<-,thick, >=stealth', domain={{180 + \phi+\margin}}:{{180 + \xi-\margin}}] plot ({\R * cos(180/\n + \phi)+ \radius*cos(\x)}, {\R * sin(180/\n + \phi ) + \radius*sin(\x)});
\draw[inner sep=0pt] 
({\R * cos(180/\n + \phi)+ \radius*cos(180 + \phi+3.5*\margin)}, {\R * sin(180/\n + \phi ) + \radius*sin(180 + \phi+3.5*\margin)})
 node (y\s) {};
\draw[inner sep=0pt] 
({\R * cos(180/\n + \phi)+ \radius*cos(180 + \xi-3.5*\margin)}, {\R * sin(180/\n + \phi ) + \radius*sin(180 + \xi-3.5*\margin)})
 node (x\s) {};
}
  \node[inner sep=0pt] (z0) at (0:0) {};
\foreach \s in {6}
{
\def \phi {-120+360/\n * (\s - 1)};
\def \psi {-120+360/\n * (\s - 0.5)};
\def \xi {-120+360/\n * (\s)};
\def \R {2*\radius*cos(180/\n)};
  \node[rectangle, inner sep=1pt] (c\s) at ({\phi}:\radius) {$\bt$};
  \draw[thick,dotted, >=stealth', domain={{\phi+\margin}}:{{\xi-\margin}}]  plot  
({\radius*cos(\x)}, {\radius*sin(\x)})
;
\draw[blue, inner sep=0pt] ({\radius*cos(\phi-3*\margin)}, {\radius*sin(\phi-3*\margin)})  node (a\s) {};
\draw[black, inner sep=0pt] ({\radius*cos(\phi+3*\margin)}, {\radius*sin(\phi+3*\margin)})  node (b\s) {};
}
\foreach \s in {2,3,4,5}
{
\draw[red, thick, densely dotted] (a\s) edge[bend right=75,  looseness = 1.5] (b\s);
}
\foreach \s/\t in {1/2,2/3,3/4,4/5}
{
\draw[red, thick, densely dotted] (x\t) edge[bend right=60,  looseness = 1] (y\s);
}
\path (z0) -- node[inner sep=0pt,pos=0.5] (r1) {$b_4$} 
node[inner sep=0pt,pos=1.05] (s1) {$a_4$} (z1);
\path (z0) -- node[inner sep=0pt,pos=0.5] (r2) {$b_3$}
node[inner sep=0pt,pos=1.1] (s2) {$a_3$} (z2);
\path (z0) -- node[inner sep=0pt,pos=0.5] (r3) {$b_2$} 
node[inner sep=0pt,pos=1.1] (s3) {$a_2$}(z3);
\path (z0) -- node[inner sep=0pt,pos=0.5] (r4) {$b_1$} 
node[inner sep=0pt,pos=1.05] (s4) {$a_1$}(z4);
\path (z0) -- node[inner sep=0pt,pos=0.5] (r5) {$b_n$}
node[inner sep=0pt,pos=1.1] (sn) {$a_n$} (z5);
\end{tikzpicture}
\begin{array}{c}
a_1 \, a_n = b_n \, b_1 = 0 \\
a_{j+1} a_j = b_j b_{j+1} = 0
\\
{\mathstrut_{\text{for any index }1 \ \leq \ j \  < \ n}}
\end{array}
\end{array}
$$
\end{lem}
\begin{proof}
 Let $1 \leq j \leq n$.
In the following we consider all indices modulo $n$.
For each node
 $v_j \in \UG_0$ its set of incident edges is given by
$E_{v_j} = \{ e_j, e_{j+1} \} = \{ e_{s(a)} \ | \ a\in Q_1: v_a = v_j\}$.
Let $a_j$ and $b_j\in Q_1$ denote the arrows 
such that
$v_{a_j} = v_{b_j} = v_j$, $s(a_j) = j$ and $s(b_j) = j+1$.
Then the permutation $\sigma$ exchanges $a_j$ and $b_j$,
so
$t(a_j) = s(b_j)$ and $t(b_j) = s(a_j)$.
This shows the claim.
\end{proof}
Next, we consider twisted Brauer graph algebras which have a circular graph.
The simplest example of such an 
algebra
 is given by $\kk \langle\!\langle a,b \rangle\! \rangle/ (a^2, b^2, ab + ba)$. 
This algebra  is  symmetric
 if and only if the field $\kk$ has characteristic two
\cite{Skowronski-Yamagata}*{Chapter IV, Example 2.8}.
 The last statement and its proof generalize
as follows:
\begin{lem}\label{lem:circuit-algebra}
Let $\A_n$ be a circular graph order as in Lemma \ref{lem:circuit-order}
and
$m\!: \UG_0 \to \N^+$ be some map.
Let $\AI_n = \A_n/z \A_n$
be the associated twisted Brauer graph algebra.\\
If the $\kk$-algebra
$\AI_n$
is symmetric,
then
 $n$ is even or the characteristic of $\kk$ is two.
\end{lem}
\begin{proof}
By Lemma \ref{lem:BGA-relations}
the twisted Brauer graph algebra $\AI_n$
has the same quiver
$(Q,I)_n$  as the circular graph order $\A_n$
together with 
 the additional relations
$$
(b_1 a_1)^{m_1} = -(a_n b_n)^{m_n} \quad \text{and}\quad (b_{j+1} a_{j+1})^{m_{j+1}} = - (a_j b_j)^{m_j}
$$
where we denote $m_j = m(v_j)$ for each index $1 \leq j \leq n$.\\
If $\AI_n$ is symmetric, there is a $\kk$-linear form $\phi\!: \AI_n \to \kk$ such that 
 $\phi(I) \neq 0$ for any left non-zero ideal $I$ of $\AI$ and
$\phi(xy) = \phi(yx)$ for all $x,y \in \AI$ \cite{Skowronski-Yamagata}*{Chapter IV, Theorem 2.2}.
For any index $1 \leq j \leq n$ we set $\lambda_j = \phi((b_j a_j)^{m_j})$. The vector space $I = \langle (b_1 a_1)^{m_1} \rangle_{\kk}$ is a left ideal of $\AI$, since $a I = 0$ for any arrow $a \in Q_1$. So $\lambda_1 \neq 0$.\\
Moreover, it holds that $\lambda_{1} = \phi((b_{1} a_{1})^{m_1}) = - \phi((a_n b_n)^{m_n}) = -\phi((b_n a_n)^{m_n})= -\lambda_n$, 
 and similarly $\lambda_{j+1} = - \lambda_j$ for any index $1 \leq j < n$. It follows that $\lambda_1 = -\lambda_n = (-1)^2\,\lambda_{n-1} = \ldots = (-1)^n \lambda_1$. 
Thus the number
$n$ is even or $\chr \kk = 2$.
\end{proof}
A graph 
is \emph{bipartite} if
each of its nodes can be colored in one of two colors such that each edge is incident to nodes of different colors.
It is a classical result that a graph is bipartite if and only if it has no circular subgraphs of odd length.

Our next result yields another second implication for our final characterization: 
\begin{prp}\label{prp:symmetric-twisted-BGA}
Let $\A$ be a ribbon graph order, $\kk$ its base field, $\UG$ its graph, $m: \UG_0 \to \N^+$ be any map.
As before, we set $\AI = \A/z \A$ where
$z = \sum_{a \in Q_1} c_a^{m_a}$. \\
If the $\kk$-algebra $\AI$ is symmetric, then the graph $\UG$
is bipartite
or the field $\kk$ has characteristic two.
\end{prp}
\begin{proof}
Let $\AI$ be symmetric.
We may assume that the characteristic of the field $\kk$ is not two
 and that the graph $\UG$ is not a tree.
Let $\UG'$ be any circular subgraph in $\UG$.
We have to show that its number $n$ of edges is even.\\
By Construction \ref{constr:graph}
there is a bijection between the edges of $\UG$ and the primitive idempotents of $\A$.
In particular, the $n$ edges of $\UG'$
correspond to an idempotent $e \in \A$ which is a sum of $n$ primitive idempotents.
By Lemma \ref{lem:subgraph} the idempotent subring $\A^e = e \A e$
is a ribbon graph order with
underlying graph $\UG'$.
Lemma \ref{lem:circuit-order} implies that
 $\A^e$ 
 is isomorphic to the circular graph order $\A_n$. 
Since $\AI$ is symmetric, so is its subalgebra $\AI^e =  e\AI e$ \cite{Skowronski-Yamagata}*{Chapter IV, Theorem 4.1}. 
Note that the element $z^e = eze$ lies in the center of $\A^e$ and
there are isomorphisms of $\kk$-algebras $\AI^e \cong \A^e /z^e\, \A^e \cong 
 \AI_n$. By Lemma \ref{lem:circuit-algebra}  the number $n$ is even.
\end{proof}

\subsection{Bipartite ribbon graph orders and Brauer graph algebras}
\label{subsec:bipartite}

For the sake of completeness we give several equivalent 
characterizations of bipartite ribbon graphs one of which is due to Kauer \cite{Kauer}:
\begin{prp}\label{prp:bipartite}
Let $\A$ be a ribbon graph order, $\UG$ its graph and $(Q,I)$ its quiver. Then the following conditions are equivalent:
\begin{enumerate}[label=$({\mathsf{\arabic*}})$, ref=${\mathsf{\arabic*}}$]
\item \label{prp:bipartite1} The graph $\UG$ is bipartite, that is,
there is a map
 $\beta\!: \UG_0 \to \{+,-\}$ 
such that $\beta(v_a) \neq \beta(v_b)$
if there is an edge between the nodes $v_a$ and $v_b$.
\item \label{prp:bipartite2} There is a \emph{$\sigma$-stable} polarization $\epsilon$ of $\A$,
that is, there is a polarization $\epsilon
 \!:Q_1 \to \{+,-\}$ 
such that $\epsilon_a = \epsilon_{\sigma(a)}$ for any arrow $a \in Q_1$.
\item \label{prp:bipartite3}
Each arrow of $Q$ can be colored in one of two colors such that 
any two arrows $a,b \in Q_1$ with  $s(b) = t(a)$ have different colors if and only $b\,a \in I$.
\item \label{prp:bipartite4} The rank of the Cartan matrix $C_\A$ of the $\Rx$-order $\A$ is equal to $|\UG_0| - \mathrm{c}_\UG$, where $c_\UG$ denotes the number of connected subgraphs of the graph $\UG$.
\end{enumerate}
\end{prp}
\begin{proof}
\begin{itemize}
\item \eqref{prp:bipartite1} $\Rightarrow$ \eqref{prp:bipartite2}:
Since $\UG$ is bipartite,
there is a map
 $\beta\!: \UG_0 \to \{+,-\}$ as above.
Let $\pi: Q_1 \twoheadrightarrow \UG_0$
be the map which assigns to each arrow $a \in Q_1$ its $\sigma$-orbit $v_a$.
We claim that
the $\sigma$-stable map 
$\epsilon = \pi \circ \beta:
Q_1 \to \{+,-\}$
 is
a polarization.\\
Let $a,b \in Q_1$ be two arrows such that $s(a) = s(b)$. By definition of the graph $\UG$ there is an edge between $v_a$ and $v_b$ in $\UG$.
This implies that
$\epsilon_{a} = \beta(v_a) \neq \beta(v_b) = \epsilon_{b}$. Thus $\epsilon_{\star}$ is a $\sigma$-stable polarization of $\A$. 
\item 
By Kauer's results \cite{Kauer}*{Hauptsatz 2.15, Satz 3.12}
it holds that $\rk C_\A = |\UG_0| - \mathrm{bc}_\UG$, 
where $\mathrm{bc}$ denotes the number of bipartite connected components of the graph $\UG$.
This implies the equivalence 
\eqref{prp:bipartite1} $\Leftrightarrow$ \eqref{prp:bipartite4}.
\end{itemize}
We skip the simple proofs of the implications \eqref{prp:bipartite2} $\Rightarrow$ \eqref{prp:bipartite3}
$\Rightarrow$ \eqref{prp:bipartite1}
 since they will not be needed in the following.
\end{proof}
The next statement yields the last two implications for the
final subsection:
\begin{prp}
\label{prp:non-twisted-BGA}
Let $\A$ be a ribbon graph order such that its graph $\UG$ is bipartite or its base field $\kk$ has characteristic two. Then the following statement holds:
\begin{enumerate}
\item \label{prp:non-twisted-BGA1} There is a polarization $\epsilon$ of $\A$ which induces a trivial involution $\epsilon^\star$.
\end{enumerate}
Let $m: \UG_0 \to \N^+$ be a map on the set of $\sigma$-orbits and $\varepsilon$ be any polarization of $\A$.
We set $z = \sum_{a \in Q_1} c_a^{m(a)}$ 
and $\AI =  \A/z\A$, $w=\sum_{a \in Q_1} \varepsilon_{a} c_a^{m(a)}$
and $\BI = \A/\A w \A$.
\begin{enumerate} \setcounter{enumi}{1}
\item \label{prp:non-twisted-BGA2} If
for each $\sigma$-orbit $v \in \UG_0$
there exists an element $\lambda
\in \kk$ such that
$\lambda^{m(v)} = - 1$, then there is a $\kk$-algebra isomorphism $\AI \cong \BI$.
\end{enumerate}
\end{prp}
\begin{proof}
Concerning the first claim let us note that
for any polarization $\varepsilon$ of $\A$ it holds that $\nuu = \id_\A \ $ if and only if $\ 
\varepsilon_{\sigma(a)} \, \varepsilon_a \, a= a
$ for any arrow $a \in Q_1$.
\begin{itemize}
\item 
Assume that the field $\kk$
has
 characteristic two.
Then any polarization of $\A$ induces a trivial involution by definition.
This was also observed in
Remark \ref{rmk:polarization}.
Moreover, it holds that $z=w$ and $\AI = \BI$.
This shows \ref{prp:non-twisted-BGA1} and \ref{prp:non-twisted-BGA2}.
\item
Assume that the graph $\UG$ is bipartite. Then there is a $\sigma$-stable polarization $\epsilon$ of $\A$ by
Proposition \ref{prp:bipartite} \eqref{prp:bipartite1} $\Rightarrow$ \eqref{prp:bipartite2}. It follows that $\epsilon^{\star} = \id_\A$, so \ref{prp:non-twisted-BGA1} is true.

Now we assume also that $\kk$ has the required roots of $-1$.
We may 
write $\UG_0 = \UG_0^+ \,{\overset{\cdot}{\cup}}\, \UG_0^-$,  where $\UG_0^\pm = 
\{ v_a \in \UG_0 \ | \ \epsilon_a = \pm \}$.
We may choose a complete set of representatives $\mathcal{A} \subseteq Q_1$
of the $\sigma$-orbits in $\UG_0^-$.
For each arrow $\alpha \in \mathcal{A}$
there is an element
$\lambda_\alpha \in \kk$
such that $\lambda_\alpha^{m(\alpha)} = -1$
by assumption.\\
Let $\psi: \A \to \A$ 
be the unique $\Rx$-algebra morphism 
which preserves any idempotent of $\A$ as well as any arrow of the complement $ Q_1\backslash{\mathcal{A}}$,
and maps 
each arrow $\alpha \in \mathcal{A}$
to 
its multiple
$\lambda_\alpha \, \alpha$.
\\ Let $\beta,\gamma \in Q_1$ be two distinct arrows starting at the same vertex.  We may assume that $v_\beta \in \UG_0^+$ and $v_\gamma \in \UG_0^-$.
Then there is a unique arrow $\alpha \in\mathcal{A}$ such that $v_\alpha = v_{\gamma}$.
Note that
$\psi(c_\beta) = c_\beta$, $\psi(c_\gamma) =  - \lambda_\alpha c_\gamma$ and
$\psi(c_\beta^{m(\beta)}+c_\gamma^{m(\gamma)}) = c_\beta^{m(\beta)} - c_\gamma^{m(\gamma)}$.
By Lemma \ref{lem:BGA-relations}
the map
$\psi$ induces a $\kk$-algebra isomorphism
$\AI \cong \BI$.
This completes the proof of
statement
\ref{prp:non-twisted-BGA2}.
\qedhere
\end{itemize}
\end{proof}
We give an example of the isomorphism constructed in the proof above:
\begin{ex} 
\label{ex:non-twisted-BGA}
Let $\A$ be the completed path algebra of the
quiver $(Q,I)_n$ from Example \ref{ex:cuspidal}.
Its graph $\UG$ is a line by Example \ref{ex:line}.
The map $m \equiv 1$ yields the algebras $\AI \cong \kk Q/(I+J_+)$ and $\BI \cong \kk Q/(I+J_-)$,
where $(Q,I+J_\pm)$ is the quiver
$$
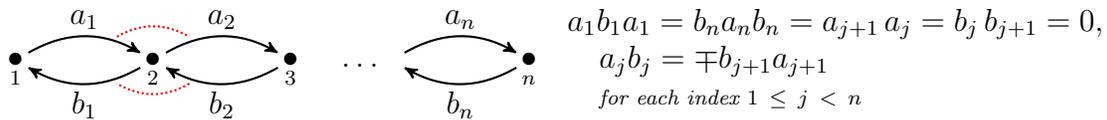
\begin{tikzcd}[nodes={
baseline=0pt,
inner sep=1pt}, 
every label/.append style={font=\normalsize, color=black},
column sep=1.45cm, cells={shape=circle},
arrows={->,thick, >=stealth'},
]
\bullet \ar[r, bend left, "a_1",
""{name=a1, near start, inner sep =-1pt},
""{name=a2, near end,  inner sep =-1pt}
] 
& \bullet \ar[r, bend left, "a_2", 
""{name=c1, near start, inner sep =-1pt},
""{name=c2, near end, swap, inner sep =-1pt}
] 
 \ar[l, bend left, "b_1",
""{name=b1, near start, inner sep =-1pt},
""{name=b2, near end, inner sep =-1pt}
] 
& {\bullet}
  \ar[l, bend left, "b_2",
""{name=d1, near start, swap, inner sep =-1pt},
""{name=d2, near end, inner sep =-1pt}
]  
&[-1cm] \ldots
&[-1.5cm]
 \phantom{\bullet}   \ar[r, bend left, "a_n", 
""{name=e1, near start, inner sep =-1pt},
""{name=e2, near end,  inner sep =-1pt}
] 
&
 \bullet  \ar[l, bend left, "b_n",
""{name=f1, near start,  inner sep =-1pt},
""{name=f2, near end, inner sep =-1pt}
]  
 \\[-1.05cm]
{\scriptstyle 1} & {\scriptstyle 2} & {\scriptstyle 3} & & & {\scriptstyle n}
\arrow[from=a2, to=c1, densely dotted, red, bend left,-]
\arrow[from=d2, to=b1, densely dotted, red, bend left,-]
\end{tikzcd} 
\
\begin{array}{l}
a_1 b_1  a_1 = b_n a_n b_n = 
a_{j+1} \, a_j = b_j \, b_{j+1} = 0, \\
\quad a_j b_j = \mp b_{j+1} a_{j+1}  
\\
\quad 
{\scriptstyle{\text{for each index }}}
\scriptstyle{1 \ \leq \ j \ < \ n}
\end{array}
$$
The last  proof yields
the $\kk$-algebra isomorphism $\psi\!: \AI \overset{\sim}{\longrightarrow} \BI$ such that
$a_j \mapsto (-1)^j a_j$, $\ b_j \mapsto b_j$ and $e_j \mapsto e_j$ for any index 
$1 \leq j \leq n$.
\end{ex}
In the above example
 the ribbon graph order $\A$ is a flat deformation 
of the Brauer graph algebra $\AI$
and the latter has also 
a Lie-theoretic interpretation \cite{Mazorchuk-Stroppel}.

\begin{rmk}
If $n = 3$, $m \equiv 2$ and $\kk = \R$ in  Example \ref{ex:non-twisted-BGA} above, then $\AI \not\cong \BI$.
This shows that the conclusion of Proposition \ref{prp:non-twisted-BGA} \ref{prp:non-twisted-BGA2} may be wrong in general.
\end{rmk}

\subsection{Characterization of symmetric ribbon graph orders}\label{subsec:final}
Finally, we 
may characterize symmetric ribbon graph orders as follows:
\begin{thm}\label{thm:symmetry}
Let $\A$  be a ribbon graph order, $\kk$ its base field, $\UG$ the graph
of a complete gentle quiver of $\A$, $m\!: \UG_0 \to \N^+$ be any map on the set of its $\sigma$-orbits
and
$\varepsilon$ be any polarization of $\A$.\\
We view $\A$ as an $\Rx$-order via the map $\Rx \to \A$, $t \mapsto z=\sum_{a \in Q_1} c_a^{m(v_a)}$
and set $w = \sum_{a \in Q_1} \varepsilon_a c_a^{m(v_a)}$.
Consider the following six conditions:
\begin{enumerate}[label=$({\mathsf{\arabic*}})$, ref=${\mathsf{\arabic*}}$]
\item \label{thm:symmetry1} The ribbon graph order $\A$ is symmetric.
\item \label{thm:symmetry2} 
The twisted Brauer graph algebra $\AI = \A/z\A$ is symmetric.
\item \label{thm:symmetry3}
The graph $\UG$ is bipartite or
the field $\kk$ has characteristic two.
\item \label{thm:symmetry4} There is a polarization $\epsilon$ of $\A$ such that its involution $\epsilon^{\star}$ is trivial. 
\item \label{thm:symmetry6} The algebra $\AI$ is isomorphic to some Brauer graph algebra.
\item \label{thm:symmetry5} The algebra $\AI$ is isomorphic to the Brauer graph algebra $\BI = \A/\A w \A$.
\end{enumerate}
Then the following statements hold:
\begin{enumerate}
\item 
The implications 
$\eqref{thm:symmetry1}\Leftrightarrow
\eqref{thm:symmetry2} \Leftrightarrow
\eqref{thm:symmetry3} \Leftrightarrow
\eqref{thm:symmetry4} \Leftarrow
\eqref{thm:symmetry6}
\Leftarrow
\eqref{thm:symmetry5}
$
are true.
\item
If the field $\kk$ has a root of the polynomial $x^{m(v)} + 1$ for each node $v \in \UG_0$,
 then
all six conditions are equivalent.
\end{enumerate}
\end{thm}
\begin{proof}
\begin{itemize}
\item 
The implications
\eqref{thm:symmetry1} $\Rightarrow$ \eqref{thm:symmetry2} $\Leftarrow$ \eqref{thm:symmetry6}
  follow
from Theorem
 \ref{thm:twisted-BGA} \ref{thm:twisted-BGA3}.
\item
\eqref{thm:symmetry2} $\Rightarrow$ \eqref{thm:symmetry3} $\Rightarrow$
\eqref{thm:symmetry4}
were shown in Propositions
\ref{prp:symmetric-twisted-BGA}
and \ref{prp:non-twisted-BGA} \ref{prp:non-twisted-BGA1}.
\item \eqref{thm:symmetry4} $\Rightarrow$ \eqref{thm:symmetry1}: 
Let $\epsilon$ be a polarization of $\A$ such that $\epsilon^{\star} = \id_\A$.
To show that $\A$ is symmetric, we may assume that $m \equiv 1$ by Corollary \ref{cor:symmetry}.
In this case, there is an isomorphism $\A = \A_{\epsilon^{\star}}   \cong \omegac$ of $\A$-bimodules by Theorem \ref{thm:A}. 
\item 
If  $\kk$ has the required roots of $-1$, then 
\eqref{thm:symmetry4} $\Rightarrow$ \eqref{thm:symmetry5}
by Proposition \ref{prp:non-twisted-BGA}
\ref{prp:non-twisted-BGA2}.
\qedhere
\end{itemize}
\end{proof}
At last, we give some examples and remarks on the last theorem.
\begin{ex}
\begin{enumerate}
\item
The  graph of the ribbon graph order
$\A = \kk \langle \! \langle a,b \rangle\! \rangle / (a^2,b^2)$
is a loop.
By Theorem \ref{thm:symmetry} 
the order 
$\A$
is symmetric if and only if $\chr \kk = 2$.
\item More generally, a circuit graph order $\A_n$ with $n$ vertices from Lemma \ref{lem:circuit-order}
is symmetric if and only if $n$ is even or $\chr \kk = 2$.
\end{enumerate}
\end{ex}

\begin{rmk}
There
 are direct proofs for the implications \eqref{thm:symmetry1} $\Leftrightarrow$
\eqref{thm:symmetry3}
$\Leftrightarrow$ 
\eqref{thm:symmetry4} in Theorem \ref{thm:symmetry}.
In particular, the introduction of twisted Brauer graph algebras is not  necessary to prove an order-theoretic version of the last theorem.
%
\end{rmk}

Concerning  quotients of ribbon graph orders let us note the following:
\begin{rmk}
A $\kk$-algebra
is \emph{special biserial} if and only if it is some quotient of some ribbon graph order. 
If the field $\kk$ is algebraically closed, 
then
Brauer graph $\kk$-algebras 
are precisely the finite-dimensional special biserial symmetric $\kk$-algebras
by a result of Schroll \cite{Schroll}*{Theorem 1.1}.
\end{rmk}
Let $\mathcal{A}$, $\mathcal{B}$ and $\mathcal{C}$ denote the isomorphism classes of twisted Brauer graph algebras,  usual Brauer graph algebras respectively symmetric special biserial algebras over some fixed field $\kk$.
These three classes of algebras are related as follows:
$$
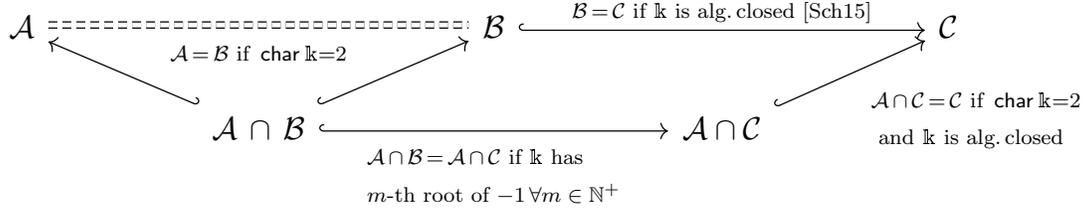
\begin{tikzcd}[ampersand replacement=\&, column sep=2cm]
{\mathcal{A}} \ar[dashed,equal,yshift=2pt]{rr}[swap,yshift=-5pt]{
  \mathcal{A} \, = \, \mathcal{B} \text{ if } \chr \kk = 2 } 
\ar[hookleftarrow]{rd}
\& \&
\mathcal{B}  \ar[hookrightarrow]{rr}{\mathcal{B} \, = \, \mathcal{C} \text{ if }\kk \text{ is alg.\,closed \cite{Schroll}}   } \& \& \mathcal{C} \\
\&  \mathcal{A}  \, \cap  \, \mathcal{B} \ar[hookrightarrow]{ru}
\ar[hookrightarrow]{rr}[swap]{  
\begin{array}{l} \mathcal{A} \, \cap \, \mathcal{B} \,  =  \, \mathcal{A} \, \cap \, \mathcal{C} \text{ if $\kk$ has }  \\
  \text{$m$-th root of $-1 \, \forall m \in \N^+$} 
\end{array}}
 \&
\& 
\mathcal{A} \cap \mathcal{C} \ar[hookrightarrow]{ru}[swap]{
\begin{array}{l}
\mathcal{A}  \, \cap \, \mathcal{C}   
 \, =  \, \mathcal{C} \text{ if } \chr \kk = 2 \\
 \text{ and } \kk \text{ is alg.\,closed} \end{array} }
\end{tikzcd}
$$

Finally, let us recall two equivalent notions related to ribbon graphs:
\begin{rmk}\label{rmk:dessin}
A \emph{dimer model} or a \emph{dessin d'enfants}  can be viewed as 
a bipartite ribbon graph together with a bicoloring of the nodes such that any edge connects nodes of different colors
$($see \cite{Lando-Zvonkin}*{Sections 1.5, 2.1} respectively
 \cite{Bocklandt}$)$.\\
By the last Theorem, any dimer model 
together with some choice of multiplicities on its nodes
give rise to
the $0$-Calabi-Yau category $\Perf \AI$ of a symmetric $\kk$-algebra $\AI$ and the $1$-Calabi-Yau category $\Perffd \A$ of a symmetric $\Rx$-order $\A$.
\\
We refer to 
Broomhead's work \cite{Broomhead}
for results and a survey
on the relationship between dimer models and certain $3$-Calabi-Yau categories.
\end{rmk}
In summary,
Theorem \ref{thm:symmetry} suggests
that twisted Brauer graph algebras are closer to ribbon graph orders
in terms of their homological properties than the symmetric Brauer graph algebras.

In a subsequent paper
\cite{Gnedin2},
we will study further homological features of ribbon graph orders and twisted Brauer graph algebras from an order-theoretic perspective.

\section*{Acknowledgement}
\thanks{The author would like to thank  William Crawley-Boevey for pointing out the reference \cite{Butler-Ringel}.
He is grateful to
Steffen König and Markus Reineke for helpful advice and productive discussions. 
}

\bibliographystyle{alpha}
\bibliography{References}

\begin{bibdiv}
\begin{biblist}

\bib{Brenner-Butler-King}{article}{
      author={Brenner, Sheila},
      author={Butler, Michael C.~R.},
      author={King, Alastair~D.},
       title={Periodic algebras which are almost {K}oszul},
        date={2002},
        ISSN={1386-923X},
     journal={Algebr. Represent. Theory},
      volume={5},
      number={4},
       pages={331\ndash 367},
         url={https://doi.org/10.1023/A:1020146502185},
      review={\MR{1930968}},
}

\bib{Burban-Drozd}{article}{
      author={Burban, Igor},
      author={Drozd, {Yu}riy},
       title={Derived categories of nodal algebras},
        date={2004},
        ISSN={0021-8693},
     journal={J. Algebra},
      volume={272},
      number={1},
       pages={46\ndash 94},
         url={https://doi.org/10.1016/j.jalgebra.2003.07.025},
      review={\MR{2029026}},
}

\bib{Burban-Drozd18}{misc}{
      author={{Burban}, Igor},
      author={{Drozd}, Yuriy},
       title={{Non-commutative nodal curves and derived tame algebras}},
        date={2018},
        note={\href{https://arxiv.org/abs/1805.05174v2}{\tt arXiv:1805.05174v2
  [math.AG]}},
}

\bib{Bocklandt}{article}{
      author={Bocklandt, Raf},
       title={A dimer {ABC}},
        date={2016},
        ISSN={0024-6093},
     journal={Bull. Lond. Math. Soc.},
      volume={48},
      number={3},
       pages={387\ndash 451},
         url={https://doi.org/10.1112/blms/bdv101},
      review={\MR{3509904}},
}

\bib{Butler-Ringel}{article}{
      author={Butler, M. C.~R.},
      author={Ringel, Claus~Michael},
       title={Auslander-{R}eiten sequences with few middle terms and
  applications to string algebras},
        date={1987},
        ISSN={0092-7872},
     journal={Comm. Algebra},
      volume={15},
      number={1-2},
       pages={145\ndash 179},
         url={https://doi.org/10.1080/00927878708823416},
      review={\MR{876976}},
}

\bib{Broomhead}{article}{
      author={Broomhead, Nathan},
       title={Dimer models and {C}alabi-{Y}au algebras},
        date={2011},
     journal={Memoirs of the AMS},
      volume={215},
       pages={153\ndash 256},
}

\bib{Ginzburg}{misc}{
      author={{Ginzburg}, Victor},
       title={{Calabi-Yau algebras}},
        date={2006},
        note={\href{https://arxiv.org/abs/math/0612139v3}{\tt arXiv:0612139v3
  [math.AG]}},
}

\bib{Goto-Nishida}{article}{
      author={Goto, Shiro},
      author={Nishida, Kenji},
       title={Towards a theory of {B}ass numbers with application to
  {G}orenstein algebras},
        date={2002},
        ISSN={0010-1354},
     journal={Colloq. Math.},
      volume={91},
      number={2},
       pages={191\ndash 253},
         url={https://doi.org/10.4064/cm91-2-4},
      review={\MR{1898632}},
}

\bib{Gnedin2}{misc}{
      author={{Gnedin}, Wassilij},
       title={{Ribbon graph orders are preserved by derived equivalences}},
        note={in preparation},
}

\bib{Grantcharov-Serganova}{article}{
      author={Grantcharov, Dimitar},
      author={Serganova, Vera},
       title={Cuspidal representations of {${\mathfrak{sl}}(n+1)$}},
        date={2010},
        ISSN={0001-8708},
     journal={Adv. Math.},
      volume={224},
      number={4},
       pages={1517\ndash 1547},
         url={https://doi.org/10.1016/j.aim.2009.12.024},
      review={\MR{2646303}},
}

\bib{Happel}{book}{
      author={Happel, Dieter},
       title={Triangulated categories in the representation theory of
  finite-dimensional algebras},
      series={London Mathematical Society Lecture Note Series},
   publisher={Cambridge University Press, Cambridge},
        date={1988},
      volume={119},
        ISBN={0-521-33922-7},
         url={https://doi.org/10.1017/CBO9780511629228},
      review={\MR{935124}},
}

\bib{Herschend-Iyama}{article}{
      author={Herschend, Martin},
      author={Iyama, Osamu},
       title={{$n$}-representation-finite algebras and twisted fractionally
  {C}alabi-{Y}au algebras},
        date={2011},
        ISSN={0024-6093},
     journal={Bull. Lond. Math. Soc.},
      volume={43},
      number={3},
       pages={449\ndash 466},
         url={https://doi.org/10.1112/blms/bdq101},
      review={\MR{2820136}},
}

\bib{HKK}{article}{
      author={Haiden, F.},
      author={Katzarkov, L.},
      author={Kontsevich, M.},
       title={Flat surfaces and stability structures},
        date={2017},
        ISSN={0073-8301},
     journal={Publ. Math. Inst. Hautes \'{E}tudes Sci.},
      volume={126},
       pages={247\ndash 318},
         url={https://doi.org/10.1007/s10240-017-0095-y},
      review={\MR{3735868}},
}

\bib{Iyama-Reiten}{article}{
      author={Iyama, Osamu},
      author={Reiten, Idun},
       title={Fomin-{Z}elevinsky mutation and tilting modules over
  {C}alabi-{Y}au algebras},
        date={2008},
        ISSN={0002-9327},
     journal={Amer. J. Math.},
      volume={130},
      number={4},
       pages={1087\ndash 1149},
         url={https://doi.org/10.1353/ajm.0.0011},
      review={\MR{2427009}},
}

\bib{Kauer}{thesis}{
      author={Kauer, Michael},
       title={Derivierte {Ä}quivalenz von {G}raphordnungen und
  {G}raphalgebren},
        type={Ph.D. Thesis},
   publisher={Shaker Verlag, Aachen},
        date={1998},
        note={University of Stuttgart},
}

\bib{Keller}{incollection}{
      author={Keller, Bernhard},
       title={Calabi-{Y}au triangulated categories},
        date={2008},
   booktitle={Trends in representation theory of algebras and related topics},
      series={EMS Ser. Congr. Rep.},
   publisher={Eur. Math. Soc., Z\"{u}rich},
       pages={467\ndash 489},
         url={https://doi.org/10.4171/062-1/11},
      review={\MR{2484733}},
}

\bib{Kauer-Roggenkamp}{article}{
      author={Kauer, Michael},
      author={Roggenkamp, Klaus~W.},
       title={Higher-dimensional orders, graph-orders, and derived
  equivalences},
        date={2001},
        ISSN={0022-4049},
     journal={J. Pure Appl. Algebra},
      volume={155},
      number={2-3},
       pages={181\ndash 202},
         url={https://doi.org/10.1016/S0022-4049(99)90107-X},
      review={\MR{1801414}},
}

\bib{Ladkani2012}{misc}{
      author={{Ladkani}, Sefi},
       title={{On Jacobian algebras from closed surfaces}},
        date={2012},
        note={\href{https://arxiv.org/abs/1207.3778v1}{\tt arXiv:1207.3778v1}},
}

\bib{Ladkani2017}{incollection}{
      author={Ladkani, Sefi},
       title={From groups to clusters},
        date={2017},
   booktitle={Representation theory---current trends and perspectives},
      series={EMS Ser. Congr. Rep.},
   publisher={Eur. Math. Soc., Z\"{u}rich},
       pages={427\ndash 500},
      review={\MR{3644801}},
}

\bib{LP}{misc}{
      author={{Lekili}, Yanki},
      author={{Polishchuk}, Alexander},
       title={{Derived equivalences of gentle algebras via Fukaya categories}},
        date={2018},
        note={\href{https://arxiv.org/abs/1801.06370v4}{\tt arXiv:1801.06370v4
  [math.SG]}},
}

\bib{Lando-Zvonkin}{book}{
      author={Lando, Sergei~K.},
      author={Zvonkin, Alexander~K.},
       title={Graphs on surfaces and their applications},
      series={Encyclopaedia of Mathematical Sciences},
   publisher={Springer-Verlag, Berlin},
        date={2004},
      volume={141},
        ISBN={3-540-00203-0},
         url={https://doi.org/10.1007/978-3-540-38361-1},
      review={\MR{2036721}},
}

\bib{Mazorchuk-Stroppel}{article}{
      author={Mazorchuk, Volodymyr},
      author={Stroppel, Catharina},
       title={Cuspidal {$\germ{sl}_n$}-modules and deformations of certain
  {B}rauer tree algebras},
        date={2011},
        ISSN={0001-8708},
     journal={Adv. Math.},
      volume={228},
      number={2},
       pages={1008\ndash 1042},
         url={https://doi.org/10.1016/j.aim.2011.06.005},
      review={\MR{2822216}},
}

\bib{OPS}{misc}{
      author={{Opper}, Sebastian},
      author={{Plamondon}, Pierre-Guy},
      author={{Schroll}, Sibylle},
       title={{A geometric model for the derived category of gentle algebras}},
        date={2018},
        note={\href{https://arxiv.org/abs/1801.09659v5}{\tt arXiv:1801.09659v5
  [math.RT]}},
}

\bib{PPP2}{misc}{
      author={{Palu}, Yann},
      author={{Pilaud}, Vincent},
      author={{Plamondon}, Pierre-Guy},
       title={{Non-kissing and non-crossing complexes for locally gentle
  algebras}},
        date={2018},
        note={\href{https://arxiv.org/abs/1807.04730}{\tt arXiv:1807.04730
  [math.CO]}},
}

\bib{Ringel}{misc}{
      author={Ringel, C.M.},
       title={Representations of string algebras},
        date={2011},
        note={\href{https://www.math.uni-bielefeld.de/~sek/shanghai/}{Shanghai
  Lecture Notes}},
}

\bib{Ringel-Roggenkamp}{article}{
      author={Ringel, Claus~Michael},
      author={Roggenkamp, Klaus~W.},
       title={Diagrammatic methods in the representation theory of orders},
        date={1979},
        ISSN={0021-8693},
     journal={J. Algebra},
      volume={60},
      number={1},
       pages={11\ndash 42},
         url={https://doi.org/10.1016/0021-8693(79)90106-6},
      review={\MR{549096}},
}

\bib{Schroll}{article}{
      author={Schroll, Sibylle},
       title={Trivial extensions of gentle algebras and {B}rauer graph
  algebras},
        date={2015},
        ISSN={0021-8693},
     journal={J. Algebra},
      volume={444},
       pages={183\ndash 200},
         url={https://doi.org/10.1016/j.jalgebra.2015.07.037},
      review={\MR{3406174}},
}

\bib{Schroll18}{incollection}{
      author={Schroll, Sibylle},
       title={Brauer graph algebras: a survey on {B}rauer graph algebras,
  associated gentle algebras and their connections to cluster theory},
        date={2018},
   booktitle={Homological methods, representation theory, and cluster
  algebras},
      series={CRM Short Courses},
   publisher={Springer, Cham},
       pages={177\ndash 223},
      review={\MR{3823391}},
}

\bib{Skowronski-Yamagata}{book}{
      author={Skowro\'{n}ski, Andrzej},
      author={Yamagata, Kunio},
       title={Frobenius algebras {I}. {B}asic representation theory},
      series={EMS Textbooks in Mathematics},
   publisher={EMS, Z\"{u}rich},
        date={2011},
        ISBN={978-3-03719-102-6},
         url={https://doi.org/10.4171/102},
      review={\MR{2894798}},
}

\bib{van_den_Bergh}{incollection}{
      author={van~den Bergh, Michel},
       title={Non-commutative crepant resolutions},
        date={2004},
   booktitle={The legacy of {N}iels {H}enrik {A}bel},
   publisher={Springer, Berlin},
       pages={749\ndash 770},
      review={\MR{2077594}},
}

\end{biblist}
\end{bibdiv}

\end{document}